\documentclass{amsart}
\usepackage{latexsym, amssymb, amsmath, amsthm}
\usepackage[T1]{fontenc}
\usepackage{lmodern}
\usepackage{enumerate}
\usepackage{mathabx}
\usepackage{csquotes}
\usepackage[mathscr]{euscript}
\usepackage{parskip}
\usepackage{bussproofs}
\usepackage{xcolor}

\usepackage{pifont}

\usepackage{bm}
\usepackage{tikz}
\usetikzlibrary{automata, arrows.meta, positioning}

\usepackage{mathtools, mathrsfs, aliascnt, enumitem, adjustbox}
\usepackage{hyperref, breakurl}
\makeatletter
\def\LT@start{%
\let\LT@start\endgraf
\endgraf\penalty\z@\vskip\LTpre
\dimen@\pagetotal
\advance\dimen@ \ht\ifvoid\LT@firsthead\LT@head\else\LT@firsthead\fi
\advance\dimen@ \dp\ifvoid\LT@firsthead\LT@head\else\LT@firsthead\fi
\advance\dimen@ \ht\LT@foot
\dimen@ii\vfuzz
\vfuzz\maxdimen
\setbox\tw@\copy\z@
\setbox\tw@\vsplit\tw@ to \ht\@arstrutbox
\setbox\tw@\vbox{\unvbox\tw@}%
\vfuzz\dimen@ii
\advance\dimen@ \ht
\ifdim\ht\@arstrutbox>\ht\tw@\@arstrutbox\else\tw@\fi
\advance\dimen@\dp
\ifdim\dp\@arstrutbox>\dp\tw@\@arstrutbox\else\tw@\fi
\advance\dimen@ -\pagegoal
\ifdim \dimen@>\z@\unskip\vfil\break\fi
\global\@colroom\@colht
\ifvoid\LT@foot\else
\advance\vsize-\ht\LT@foot
\global\advance\@colroom-\ht\LT@foot
\dimen@\pagegoal\advance\dimen@-\ht\LT@foot\pagegoal\dimen@
\maxdepth\z@
\fi
\ifvoid\LT@firsthead\copy\LT@head\else\box\LT@firsthead\fi
\output{\LT@output}}
\makeatother

\usepackage[textsize=footnotesize]{todonotes}

\definecolor{dblue}{rgb}{0,0,0.70}
\hypersetup{
	unicode=true,
	colorlinks=true,
	citecolor=dblue,
	linkcolor=dblue,
	anchorcolor=dblue
}

\makeatletter
\expandafter\g@addto@macro\csname th@plain\endcsname{%
	\thm@notefont{\bfseries}
}%
\expandafter\g@addto@macro\csname th@remark\endcsname{%
	\thm@headfont{\bfseries}
}%
\makeatother

\newtheorem{theorem}{Theorem}[section]	
\newtheorem*{theorem*}{Theorem}

\newaliascnt{lemma}{theorem}
\newtheorem{lemma}[lemma]{Lemma}
\aliascntresetthe{lemma}
\newtheorem*{lemma*}{Lemma}

\newaliascnt{proposition}{theorem}
\newtheorem{proposition}[proposition]{Proposition}
\aliascntresetthe{proposition}

\newaliascnt{corollary}{theorem}
\newtheorem{corollary}[corollary]{Corollary}
\aliascntresetthe{corollary}

\theoremstyle{remark}

\newaliascnt{remark}{theorem}
\newtheorem{remark}[remark]{Remark}
\aliascntresetthe{remark}
\newaliascnt{question}{theorem}

\aliascntresetthe{question}

\newtheorem*{question*}{Question}

\newaliascnt{definition}{theorem}
\newtheorem{definition}[definition]{Definition}
\aliascntresetthe{definition}

\newaliascnt{example}{theorem}

\aliascntresetthe{example}

\newaliascnt{convention}{theorem}

\aliascntresetthe{convention}

\renewcommand{\phi}{\varphi}
\newcommand{\KP}{\mathsf{KP}}

\newcommand{\ZFC}{\mathsf{ZFC}}
\newcommand{\height}{\mathsf{height}}
\newcommand{\RFN}{\mathsf{RFN}}
\renewcommand{\Pr}{\mathsf{Pr}}
\newcommand{\Ord}{\mathsf{Ord}}
\newcommand{\Comp}{\operatorname{\mathsf{Comp}}}
\newcommand{\Tr}{\mathrm{Tr}}
\newcommand{\Exists}{\operatorname{\mathsf{Exists}}}

\newcommand{\bfalpha}{{\boldsymbol{\alpha}}}
\newcommand{\bfbeta}{{\boldsymbol{\beta}}}
\newcommand{\bfgamma}{{\boldsymbol{\gamma}}}
\newcommand{\bfdelta}{{\boldsymbol{\delta}}}

\newcommand{\barrepr}{{\text{-}\mathsf{repr}}}

\newcommand{\TC}{\mathsf{TC}}

\title{Generalized ordinal analysis and reflection principles in set theory}
\author{Hanul Jeon}
\author{James Walsh}
\thanks{Thanks to Juan Pablo Aguilera, Wolfram Pohlers, and Shervin Sorouri for helpful discussions.}
\address{Department of Mathematics, Cornell University}
\email{hj344@cornell.edu}
\address{Department of Philosophy, New York University}
\email{jmw534@nyu.edu}

\begin{document}

\maketitle

\begin{abstract}
It is widely claimed that the natural axiom systems---including the large cardinal axioms---form a well-ordered hierarchy. Yet, as is well-known, it is possible to exhibit non-linearity and ill-foundedness by means of \emph{ad hoc} constructions. In this paper we formulate notions of proof-theoretic strength based on set-theoretic reflection principles. We prove that they coincide with orderings on theories given by the generalized ordinal analysis of Pohlers. Accordingly, these notions of proof-theoretic strength engender genuinely well-ordered hierarchies. The reflection principles considered in this paper are formulated relative to G\"{o}del's constructible universe; we conclude with generalizations to other inner models.
\end{abstract}

\section{Introduction}

One of the central problems in the foundations of mathematics is the apparent well-ordering of natural theories by various measures of proof-theoretic strength, such as consistency strength. Taken in full generality, these orderings are neither linear nor well-founded. Yet, if one restricts one's attention to sufficiently ``natural'' theories, neither non-linearity nor ill-foundedness arise.\footnote{At least, this seems to be the majority opinion; see, for instance, discussions in \cite{koellner2010independence, montalban2019martin, shelah2003logical, steel2014godel}. There have been dissenting voices, however; see  \cite{hamkins2022nonlinearity, hauser2014strong}.} Of particular interest is the apparent well-ordering of large cardinal axioms. Koellner writes that ``large cardinal axioms \dots lie at the heart of the remarkable empirical fact that natural theories from completely distinct domains can be compared'' \cite{koellner2010independence}. What explains the apparent well-ordering of the natural mathematical theories such as the large cardinal axioms? The non-mathematical quantification over ``natural'' theories makes it difficult to approach this question mathematically. Though Shelah says that explaining ``the phenomena of linearity of consistency strength (and of large cardinal properties)'' is a ``logical dream,'' he writes that ``having an answer may well be disputable; it may have no solution or several'' (\cite{shelah2003logical} pp. 215--216).

Steel has emphasized that---for natural theories---the consistency strength ordering coincides with the ordering given by the inclusion relation on theorems from various pointclasses. He writes:
\begin{displayquote}[Steel \cite{steel2014godel} p. 159]
If $T$ and $U$ are natural extensions of $\mathsf{ZFC}$, then 
\begin{equation*}
    T \leq_\mathsf{Con}U \Leftrightarrow (\Pi^0_1)_T \subseteq (\Pi^0_1)_U \Leftrightarrow (\Pi^0_\omega)_T \subseteq (\Pi^0_\omega)_U.
\end{equation*}
Thus the well-ordering of natural consistency strengths corresponds to a well-ordering by the inclusion of theories of the natural numbers. There is no divergence at the arithmetic level, if one climbs the consistency strength hierarchy in any natural way we know of\dots. Natural ways of climbing the consistency strength hierarchy do not diverge in their consequences for the reals\dots. Let $T, U$ be natural theories of consistency strength at least that of ``there are infinitely many Woodin cardinals''; then either $(\Pi^1_\omega)_T\subseteq(\Pi^1_\omega)_U$ or $(\Pi^1_\omega)_U\subseteq (\Pi^1_\omega)_T$.
\end{displayquote}

In previous work of the second-named author \cite{walsh2023characterizations}, analogues of consistency strength based on $\Pi^1_1$-reflection were introduced that genuinely well-order large classes of theories. These are proven to coincide with inclusion relations on theorems from suitably chosen pointclasses. This provides an insight into Steel's remarks that removes the non-mathematical quantification over ``natural'' theories. That work took place in the setting of second-order arithmetic, which limits the applicability and generality of those results. The main goal in this paper is to explore the proof-theoretic well-ordering phenomenon in set theory, the context wherein large cardinal axioms arise. Hence, we introduce and study proof-theoretic orderings that exhibit genuine linearity and genuine well-foundedness. In particular, we formulate set-theoretic reflection principles, calibrate their strength, and show that they well-order vast swathes of theories.  

Let's introduce a few important definitions before continuing.

\begin{definition}
    The formula $F^z$ is obtained from $F$ by restricting all unrestricted quantifiers in $F$ to $z$, i.e., replacing $\forall x$ by $\forall x\in z$ and $\exists x$ by $\exists x\in z$.
\end{definition}
It follows that  $F^z$ is always a $\Delta_0$ formula. 

We will be particularly interested in formulas whose quantifiers are restricted to $L_\alpha$. We will be proving general results about such formulas within formal systems. To prove such results, $L_\alpha$ must be definable, whence $\alpha$ must have a good representation whose properties can be established over a weak set theory like $\KP$. In this paper, we adopt the convention that $\bfalpha$ is a representation of the ordinal $\alpha$ satisfying various criteria (see Definition \ref{representation} and the discussion afterwards for details). Moreover, we work with theories that prove the $\Sigma$-recursion principle so that, given a representation $\bfalpha$ of an ordinal, we may define $L_\bfalpha$ within the theory.

\begin{definition}
The formula $G$ is $\Pi^\bfalpha_n$ ($\Sigma^\bfalpha_n$) if $G$ has the form $F^{L_\bfalpha}$ for some $\Pi_n$ ($\Sigma_n$) formula $F$. 
\end{definition}
Finally, in what follows, $\mathsf{KP}_1$ is the fragment of $\mathsf{KP}$ that results from restricting set induction to $(\Sigma_1\cup\Pi_1)$ formulas.

The reflection principles we will be working with are those defined as follows:
$$\RFN_{\Sigma^\bfalpha_1}(T) := \forall \varphi\in \Delta_0 \Big(\Pr_T\big(\exists x \in L_\bfalpha \; \varphi(x)\big)\to L_\bfalpha\models_{\Sigma_1} \exists x \varphi(x)\Big).$$
Note that this carves out a whole class of reflection principles, one for each $\bfalpha$. Moreover, note that each of these reflection principles is a formal statement within the language of set theory. 

One of our main theorems is that the relation $\{(T,U) \mid U \vdash \RFN_{\Sigma^\bfalpha_1}(T)\}$ is well-founded when restricted to a broad class of theories.
\begin{theorem*}
    There is no sequence $\langle T_n\mid n<\omega\rangle$ of $\Sigma^\bfalpha_1$-sound, $\Pi^\bfalpha_1$-definable extensions of $\KP_1 + \Exists(\bfalpha)$ such that for each $n$, $T_n\vdash \RFN_{\Sigma^\bfalpha_1}(T_{n+1})$.
\end{theorem*}

On the positive side, note that this theorem delivers as an immediate corollary a version of G\"{o}del's second incompleteness theorem: No $\Sigma^\bfalpha_1$-sound, $\Pi^\bfalpha_1$-definable extension $T$ of $\KP_1 + \Exists(\bfalpha)$ proves $\RFN_{\Sigma^\bfalpha_1}(T)$.\footnote{G\"{o}del's second incompleteness theorem states that no $\Pi^0_1$-sound, $\Sigma^0_1$-definable extension $T$ of elementary arithmetic proves $\RFN_{\Pi^0_1}(T)$. See \cite{Walsh2022Incompleteness} for discussion.} Though $\RFN_{\Sigma^\bfalpha_1}$ is stronger than consistency, $\Pi^\bfalpha_1$-definability is a weaker condition than recursive axiomatizability; accordingly, this result is neither stronger nor weaker than G\"{o}del's theorem but incomparable with it.

On the negative side, note that this result bears on the \emph{well-foundedness} of proof-theoretic strength but not on the \emph{linearity} of proof-theoretic strength. In previous work of the second-named author \cite{walsh2023characterizations}, a proof-theoretic ordering within second-order arithmetic is introduced that coincides with the well-ordering induced by comparing proof-theoretic ordinals; this ordering is both well-founded and linear. Since $\RFN_{\Sigma^\bfalpha_1}$ is stronger than the reflection principle considered in that work, proof-theoretic ordinals are not sufficient to calibrate the strength of $\RFN_{\Sigma^\bfalpha_1}$. 

However, Pohlers has emphasized that the hierarchy of theories given by proof-theoretic ordinals is actually the first in a hierarchy of hierarchies. In particular, the standard ordinal analysis measures the strength of theories with respect to capturing properties of $\omega_1^{\mathsf{CK}}$. Pohlers introduced the following generalizations of ordinal analysis that measure the strength of a theory with respect to capturing properties of other admissible ordinals. 
\begin{definition}
    Let $T$ be a theory containing $\mathsf{KP}_1$. For a complexity class $\Gamma$ we define:
        $$|T|_{\Gamma^\bfalpha} := \mathsf{min}\{\beta \mid \text{For all $\varphi\in\Gamma$, if }T\vdash \varphi^{L_\bfalpha} \text{ then } L_\beta \vDash \varphi \}.$$
\end{definition}

This notion generalizes the standard notion of ordinal analysis for the following reason. There is a standard translation $\star$ from the language of second-order arithmetic into the language of set theory. For an $\mathcal{L}_2$ theory $T$, $$|T^\star|_{\Sigma_1^{\omega_1^{\mathsf{CK}}}}=|T|_{\mathsf{WF}},$$ where $|T|_{\mathsf{WF}}$ is the standard proof-theoretic ordinal of $T$, i.e., the supremum of the $T$-provably well-founded primitive recursive linear orders. For details see \cite{Pohlers1998Subsystems}.

Though Pohlers introduced this generalization of ordinal analysis \cite{Pohlers1998Subsystems} and developed some applications thereof \cite{Pohlers2022Performance}, the notion has not been thoroughly studied. Our second main theorem is a demonstration of the robustness of this notion. In particular, we give an abstract characterization of the $\bfalpha$ ordinal analysis partition, which identifies theories that have the same $\Sigma^\bfalpha_1$ ordinal. We demonstrate that the $\bfalpha$ ordinal analysis partition is the unique finest partition satisfying certain natural criteria.\footnote{This result generalizes a characterization of the standard ordinal analysis partition given in \cite{walsh2023characterizations}.}

Our final main theorem connects reflection principles and generalized proof-theoretic ordinals. To establish a full correspondence between reflection principles and generalized proof-theoretic ordinals, we use the notion of provability in the presence of an oracle. 

\begin{definition}
    We define \emph{provability in the presence of an oracle for $\Pi^\bfalpha_1$-truths} $\vdash^{\Pi^\bfalpha_1}$ as follows: $T\vdash^{\Pi^\bfalpha_1} \varphi$ if there is a true $\Pi^\bfalpha_1$-formula $\psi$ such that $T+\psi\vdash\varphi$.
    We define $T\le^{\Pi^\bfalpha_1}_{\RFN_{\Sigma^\bfalpha_1}}U$ if and only if
    \begin{equation*}
        \KP_1 + \Exists(\bfalpha) \vdash^{\Pi^\bfalpha_1} \RFN_{\Sigma^\bfalpha_1}(U)\to \RFN_{\Sigma^\bfalpha_1}(T).
    \end{equation*}
\end{definition}

Our final main theorem is this:
\begin{theorem*}
    Let $T$ and $U$ be $\Delta^\bfalpha_1$-definable $\Sigma^\bfalpha_1$-sound extensions of $\KP_1 + \mathsf{Exists}(\bfalpha)$. Then we have
    \begin{equation*}
        |T|_{\Sigma^\bfalpha_1}\le |U|_{\Sigma^\bfalpha_1} \iff T\le^{\Pi^\bfalpha_1}_{\RFN_{\Sigma^\bfalpha_1}}U.
    \end{equation*}
\end{theorem*}
Note that this immediately entails that the ordering $<^{\Pi^\bfalpha_1}_{\RFN_{\Sigma^\bfalpha_1}}$ induces on $\le^{\Pi^\bfalpha_1}_{\RFN_{\Sigma^\bfalpha_1}}$ equivalence classes is a well-ordering. Note, moreover, that the stringent condition of $\Delta^\bfalpha_1$-definability is necessary for the theorem; for a counter-example to the corresponding claim for $\Pi^\bfalpha_1$-definability see \cite[\textsection 3.5]{walsh2023characterizations}.

We prove all of these results assuming that $V=L$. This excludes the application of these results to theories (for instance) that imply the existence of $0^\sharp$. At the end of the paper we discuss further directions. The results proved in this paper should generalize to fine structural inner models compatible with $0^\sharp$ and beyond, assuming we limit ourselves to $\bfalpha$ ordinal analysis for small enough admissible $\bfalpha$.

Let us make a few remarks about the techniques deployed in this paper. The theorems in this paper strengthen various results in \cite{Walsh2022Incompleteness, walsh2023characterizations}. The main recurring technical tool in that work is the $\Pi^1_1$-completeness of well-foundedness, i.e., every $\Pi^1_1$ statement is $\mathsf{ACA}_0$-equivalent to the statement that a particular primitive recursive relation is well-founded. In the present paper we do not work with the complexity class $\Pi^1_1$, so this completeness theorem does not secure the results we prove. Accordingly, the techniques used in this paper are novel and specially designed for the new set-theoretic context.

In this paper, we do much with $\Sigma_1^{\omega_1^{\mathsf{CK}}}$-formulas (and with $\Sigma_1^\bfalpha$ formulas generally). Following \cite{sacks2017higher}, we may view $\Sigma_1^{\omega_1^{\mathsf{CK}}}$-definable sets as \emph{$\omega_1^{\mathsf{CK}}$-recursively enumerable} sets. Instead of converting $\Pi^1_1$ statements into well-foundedness claims, one can convert $\Sigma_1^{\omega_1^{\mathsf{CK}}}$-statements into statements about hypercomputations. Such transformations are technically not necessary for our proofs, but we include some discussion of this perspective for conceptual reasons. One of the central activities of traditional ordinal analysis is the classification of recursive functions into transfinite hierarchies and the calibration of the computational strength of theories in terms of such hierarchies. We will see that this tight connection between ordinal analysis and computational strength generalizes into the domain of $\bfalpha$ ordinal analysis. Moreover, this connection can be adumbrated in wholly precise mathematical terms.

Here is our plan for the rest of the paper. In \textsection \ref{kripke-platek} we review the basics of Kripke Platek set theory. We cover definitions of the syntactic complexity classes we will work with, as well as their partial truth definitions. Since we are proving very general results in the context of set theory, it is not sufficient to treat primitive recursive well-orders as surrogates for ordinals. Hence in \textsection \ref{reasonably-definable} we discuss the manner in which we deal with representations of ordinals. In \textsection \ref{kreisel-relations} we establish the robustness of $\bfalpha$ ordinal analysis by characterizing, in abstract terms, the partition of theories that it carves out. In \textsection \ref{alpha-recursion} we connect $\bfalpha$ ordinal analysis with $\alpha$ recursion theory. This discussion lifts the well-known connections between provably total recursive functions and ordinal analysis to the domain of $\alpha$-recursion theory. In \textsection \ref{comparison-theorem} we prove that relative $\mathsf{RFN}_{\Sigma^\bfalpha_1}$ strength in the presence of appropriate oracles coincides with $\bfalpha$ ordinal analysis. In \textsection \ref{well-foundedness-thm} we prove that $\mathsf{RFN}_{\Sigma^\bfalpha_1}$ induces a well-founded relation on a large swathe of theories, even without the presence of oracles. Finally, in \textsection \ref{discussion} we discuss the connection of these results with broader set-theoretic themes. In particular, we discuss the connection between our results and the \emph{Spector classes} introduced by Moschovakis \cite{Moschovakis1974ElementaryInd}. We also discuss how to generalize our main results to inner models beyond G\"{o}del's $L$. In particular, we discuss a generalization of the linearity and well-foundedness phenomena for $\Pi^1_3$ theorems to theories that contain $\mathsf{ZFC}$ + $\mathbf{\Delta}^1_2$ determinacy.

\section{Kripke-Platek set theory}\label{kripke-platek}
Throughout this paper, we will work with the weak set theory $\KP_1$, Kripke-Platek set theory with the Foundation schema restricted to $\Sigma_1\cup\Pi_1$-formulas. Working over any weak set theory $T$ comes with an obvious risk: Some claims that seem intuitively obvious may not be $T$-provable. In this section, we will cover some of the basic facts about $\KP_1$ that we will need in the rest of our paper.

\begin{definition}
    Consider the language $\{\in\}$. The $\Delta_0$-formulas formulas are the bounded formulas. That is, $\Delta_0$ is the least class of formulas containing atomic formulas that are closed under Booleans and bounded quantification. Following the usual convention, $\Sigma_0=\Pi_0=\Delta_0$.
    
    Now we inductively define $\Sigma_n$-formulas and $\Pi_n$-formulas as follows:
    \begin{itemize}
        \item A formula is $\Sigma_{n+1}$ if it is of the form $\exists x \psi(x,y_1,\cdots y_m)$ for some $\Pi_{n}$-formula $\psi$, and
        \item A formula is $\Pi_{n+1}$ if it is of the form $\forall x \psi(x,y_1,\cdots y_m)$ for some $\Sigma_{n}$-formula $\psi$.
    \end{itemize}

    If a theory $T$ extending $\KP_1$ proves that $\phi$ is equivalent both to a $\Sigma_n$-formula and to a $\Pi_n$-formula, then we say that $\phi$ is \emph{$T$-provably $\Delta_n$.} If the theory $T$ is clear in context, we omit $T$ and simply say that $\phi$ is $\Delta_n$.
    
    $\KP_1$ is the theory comprising the following axioms: 
    \begin{enumerate}
        \item Extensionality, Union, Pairing, Infinity.
        \item $\Delta_0$-Separation: For every $\Delta_0$-formula $\phi(x)$ and parameter $a$, the following is an axiom:
        \begin{equation*}
            \exists b \forall x [x\in b\leftrightarrow x\in a\land \phi(x)].
        \end{equation*}
        That is, $\{x\in a\mid \phi(x)\}$ is a set.
        \item $\Sigma_1\cup \Pi_1$-Foundation in the form of $\in$-Induction: For every formula $\phi(x)$, which is either $\Sigma_1$ or $\Pi_1$, the following is an axiom:
        \begin{equation*}
            [\forall y\in x \phi(y)\to \phi(x)]\to\forall x\phi(x).
        \end{equation*}
        \item $\Sigma_1$-Collection: For every $\Sigma_1$-formula $\phi(x,y)$ and parameter $a$ the following is an axiom:
        \begin{equation*}
            \forall x\in a \exists y\phi(x,y)\to \exists b\forall x\in a\exists y\in b\phi(x,y).
        \end{equation*}
    \end{enumerate}
\end{definition}

Throughout this paper, it is very important to track the complexity of formulas.
Let us start with calculations of the complexity of simple formulas:
\begin{lemma}
    Each of the following formulas is equivalent to a $\Delta_0$-formula:
    \begin{enumerate}
        \item $z=\{x,y\}$.
        \item $\{x,y\}\in z$.
        \item $z=\langle x,y\rangle$, where $\langle x,y\rangle = \{\{x\},\{x,y\}\}$ is a Kuratowski ordered pair.
        \item $z = \bigcup x$.
        \item $z = x\cup y$.
        \item $z = \varnothing$.
        \item $x$ is a transitive set.
        \item $x$ is an ordinal.
    \end{enumerate}
\end{lemma}
\begin{proof}
    We can see that the following equivalences hold:
    \begin{enumerate}
        \item $z = \{x,y\} \iff (x\in z \land y\in z) \land \forall w\in z(w=x\lor w=y)$.
        \item $\{x,y\}\in z \iff \exists w\in z (w = \{x,y\})$.
        \item $z = \langle x,y\rangle \iff (\{x\}\in z\land \{x,y\}\in z)\land \forall w\in z (w = \{x\} \lor w = \{x,y\})$.
        \item $z=\bigcup x \iff [\forall w\in z\exists y\in x (w\in y)] \land [\forall y\in x\forall w\in y (w\in z)]$.
        \item $z=x\cup y \iff [\forall w\in x(w\in z)\land \forall w\in y(w\in z)] \land [\forall w\in z(z\in x \lor z\in y)]$.
        \item $z=\varnothing \iff \forall w\in z (w\neq w)$.
        \item $\text{$x$ is transitive} \iff \forall y\in x \; \forall z\in y \; (z\in x)$.
        \item $\text{$x$ is an ordinal} \iff \text{$x$ is transitive} \land \forall y\in x \; (\text{$y$ is transitive})$.
    \end{enumerate}
    The last equivalence requires clarification:
    The classical definition of $\alpha$ being an ordinal is that $\alpha$ is transitive and $(\alpha,\in)$ is well-ordered in the sense that $(\alpha,\in)$ is linear and for every non-empty $y\subseteq \alpha$, $y$ has an $\in$-minimal element. Unfortunately, the classical definition is $\Pi_1$.
    However, we can prove that $\alpha$ is an ordinal if and only if $\alpha$ is a transitive set of transitive sets. The reader can see the proof of Lemma 6.3 and 6.4 of \cite{Jeon2022Ackermann} works over $\KP_1$, and it shows $(\alpha,\in)$ is well-ordered if and only if $\alpha$ is a transitive set of transitive sets.
\end{proof}

We will make use of many more $\Delta_0$ notions. Instead of identifying them here, we will identify them as we go. We will only ever claim that a notion is $\Delta_0$ \emph{without proof} if we trust that the reader can easily verify the claim by hand.

\begin{lemma}
    $\KP_1$ proves the adequacy of the Kuratowski definition of ordered pairs in the following sense: If $\langle x_0,x_1\rangle = \langle y_0,y_1\rangle$, then $x_0=y_0$ and $x_1=y_1$.
\end{lemma}
\begin{proof}
    Suppose that $\langle x_0,x_1\rangle = \langle y_0,y_1\rangle$, that is, 
    \begin{equation*}
        \{\{x_0\},\{x_0,x_1\}\} = \{\{y_0\},\{y_0,y_1\}\}.
    \end{equation*}
    Then we have $\{x_0\}=\{y_0\}\lor \{x_0\}=\{y_0,y_1\}$.
    Similarly, we have  $\{x_0,x_1\}=\{y_0\}\lor \{x_0,x_1\}=\{y_0,y_1\}$. Based on this, let us divide the cases:
    \begin{itemize}
        \item If $\{x_0\} = \{y_0,y_1\}$, then $x_0=y_0=y_1$. Thus we get $\{x_0,x_1\}=\{y_0\}$, so $x_1=y_0$. In sum, $x_0=x_1=y_0=y_1$.

        \item If $\{x_0,x_1\}=\{y_0\}$, then by an argument similar to the previous one, we have $x_0=x_1=y_0=y_1$.

        \item Now consider the case $\{x_0\}=\{y_0\}$ and $\{x_0,x_1\}=\{y_0,y_1\}$. Then we have $x_0=y_0$, so $\{x_0,x_1\}=\{x_0,y_1\}$. Hence $x_1\in \{x_0,y_1\}$, so either $x_1=x_0$ or $x_1=y_1$.
        But if $x_1=x_0$, then from $y_1\in \{x_0,x_1\}$ we get $x_1=y_1$. \qedhere
    \end{itemize}
\end{proof}

\begin{remark}
    The proof for the above lemmas does not use $\Sigma_1$-Collection.
\end{remark}

Our definitions of $\Sigma_n$ and $\Pi_n$ formulas is quite restrictive. They apply only to formulas in prenex normal form with alternating unbounded formulas. However, it turns out that we can \emph{normalize} (i.e., transform in $\mathsf{KP}_1$) formulas in other forms into $\Sigma_n$ or $\Pi_n$-formulas. Influenced by \cite{FriedmanLiWong2016KPalpharec} and $\mathcal{E}_n$- and $\mathcal{U}_n$-hierarchies in \cite{Jeon2022Ackermann}, let us define the following:

\begin{definition} \label{Definition: LevyFleischmann}
    Define $\underline{\Sigma}_n$ and $\underline{\Pi}_n$ recursively as follows:
	\begin{itemize}
		\item $\underline{\Sigma}_0=\underline{\Pi}_0$ is the set of all bounded formulas.
		\item $\underline{\Sigma}_{n+1}$ is the least class containing $\underline{\Pi}_n$ that is closed under $\land$, $\lor$, bounded quantification and unbounded $\exists$.
		\item $\underline{\Pi}_{n+1}$ is the least class containing $\underline{\Sigma}_n$ that is closed under $\land$, $\lor$, bounded quantification and unbounded $\forall$.
	\end{itemize}
\end{definition}
$\underline{\Sigma}_1$-formulas are also known as $\Sigma$-formulas.
Unlike the usual $\Sigma_n$ or $\Pi_n$ formulas, every formula is either one of  $\underline{\Sigma}_n$ or $\underline{\Pi}_n$ for some $n$.

\begin{definition}
    A $\underline{\Sigma}_n$ or $\underline{\Pi}_n$ formula is \emph{normalizable} if it is equivalent to a $\Sigma_n$ or a $\Pi_n$ formula respectively.
\end{definition}

$\KP_1$ proves $\underline{\Sigma}_1$-formula are normalizable. In general, we can see the following holds, and we present its proof for the reader's convenience.

\begin{lemma}[Proposition 2.3, 2.4 of \cite{FriedmanLiWong2016KPalpharec}] \label{Lemma: Formula-normalization}
    $\KP_1$ proves the following:
    Suppose that $\phi$ and $\psi$ are normalizable formulas of the same complexity.
    \begin{enumerate}
        \item If $\phi$ is $\Sigma_n$, then $\exists x\phi$ and $\exists x\in a \; \phi$ are also normalizable. Similarly, if $\phi$ is $\Pi_n$, then $\forall x\phi$ and $\forall x\in a \; \phi$ are also normalizable.
        \item $\lnot\phi$, $\phi\land\psi$, $\phi\lor\psi$ are normalizable.
        \item Assuming $\Sigma_n$-Collection, if $\phi$ is a $\Sigma_m$-formula for $m\le n$, then $\forall x\in a \; \phi$ is normalizable.
        Similarly, if $\phi$ is $\Pi_n$ and $m\le n$, then $\exists x\in a \; \phi$ is normalizable.
    \end{enumerate}
\end{lemma}

The following lemma has a central role in our proof:
\begin{lemma} \label{Lemma: TwoQuantifier-reducing}
    Working over $\KP_1$, let $\phi(x_0,x_1,v_0,v_1,\cdots, v_{n-1},\vec{z})$ be a $\Delta_0$ formula. If we set
    \begin{multline*}
        \phi'(x,v_0,\cdots, v_{n-1},\vec{z})\equiv \exists u\in x\exists y_0\in u\exists y_1\in u \\ [(x = \langle y_0,y_1\rangle) \land \phi(y_0,y_1,v_0,v_1,v_0,v_1,\cdots, v_{n-1},\vec{z})],
    \end{multline*}
    then $\phi'$ is also $\Delta_0$ and we have the following:
    \begin{enumerate}
        \item[(i)] For any $x_0$ and $x_1$, if we let $x=\langle x_0,x_1\rangle$, then
        \begin{equation} \label{Formula: TwoQuantifier-reducing 0}
            \forall \vec{z} \forall v_0\forall v_1\cdots\forall v_{n-1}\ [\phi(x_0,x_1,v_0,\cdots,v_{n-1},\vec{z})\leftrightarrow \phi'(x,v_0,\cdots,v_{n-1},\vec{z})].
        \end{equation}
        \item[(ii)]  We have
        \begin{multline*}
            \forall \vec{z} \left[\forall x_0\forall x_1\exists v_0\forall v_1\cdots \mathsf{Q} v_{n-1} \phi(x_0,x_1,v_0,v_1,\cdots,v_{n-1},\vec{z})  \right.\\ \left. \leftrightarrow \forall x\exists v_0\forall v_1\cdots \mathsf{Q} v_{n-1} \phi'(x,v_0,v_1,\cdots,v_{n-1},\vec{z})\right].
        \end{multline*}
    \end{enumerate}
\end{lemma}
\begin{proof}
    Let us prove the first claim: Suppose that $x=\langle x_0,x_1\rangle$. By taking $u=\{x_0,x_1\}$, we have
    \begin{equation*}
        \exists u\in x (x_0,x_1\in u\land x=\langle x_0,x_1\rangle).
    \end{equation*}
    Thus if $\phi(x_0,x_1,v_0,\cdots,v_{n-1})$ holds, then  $\phi'(x,v_0,\cdots,v_{n-1})$ follows. 
    
    Conversely, suppose that $\phi'(x,v_0,\cdots,v_{n-1})$ holds and that for some $u\in x$ and some $y_0,y_1\in u$:
    \begin{equation*}
        x=\langle y_0,y_1\rangle\land \phi(y_0,y_1,v_0,\cdots,v_{n-1}).
    \end{equation*}

    Now let us prove the second claim by induction on $n$. The left-to-right direction is clear: For given $x_0$ and $x_1$, set $x=\langle x_0,x_1\rangle$. 
    For the other direction, suppose that for every $x$ the following holds:
    \begin{equation} \label{Formula: TwoQuantifier-reducing 1}
        \exists v_0\forall v_1\cdots \mathsf{Q} v_{n-1} \phi'(x,v_0,v_1,\cdots,v_{n-1}).
    \end{equation}
    By taking $x=\langle x_0, x_1\rangle$ from \eqref{Formula: TwoQuantifier-reducing 1} and applying \eqref{Formula: TwoQuantifier-reducing 0}, we get the following:
    \begin{equation*}
        \exists v_0\forall v_1\cdots \mathsf{Q} v_{n-1} \phi(x_0,x_1,v_0,v_1,\cdots,v_{n-1}). \qedhere 
    \end{equation*}
\end{proof}

Before stating the proof of  \autoref{Lemma: Formula-normalization}, let us observe that we can find $\phi'$ from $\phi$ that is uniform to $\vec{z}$. That is, the choice of parameters do not affect the new formula obtained by encoding variables into a single one.

\begin{proof}[Proof of \autoref{Lemma: Formula-normalization}]
    \begin{enumerate} 
        \item Suppose that $\varphi(x)$ is $\Sigma_n$. We first claim that $\exists x\varphi(x)$ is equivalent to a $\Sigma_n$-formula.
        Suppose that $\varphi(x)$ is of the form 
        \begin{equation*}
            \exists v_0\forall v_1\cdots\mathsf{Q} v_{n-1} \phi_0(x,v_0,\cdots,v_{n-1})
        \end{equation*}
        for some $\Delta_0$-formula $\phi_0$.
        By \autoref{Lemma: TwoQuantifier-reducing}, we have a $\Delta_0$-formula $\phi_0'$ satisfying the following:
        \begin{multline*}
            \forall x\forall v_0\exists v_1\cdots \overline{\mathsf{Q}} v_{n-1} \lnot\phi(x,v_0,v_1,\cdots,v_{n-1}) \\ \leftrightarrow \forall y\exists v_1\cdots \overline{\mathsf{Q}} v_{n-1} \lnot\phi_0'(y,v_1,\cdots,v_{n-1}),
        \end{multline*}
        where $\overline{\mathsf{Q}}$ is the dual of the quantifier $\mathsf{Q}$. 
        By taking the negation, we have
        \begin{multline*}
            \exists x\exists v_0\forall v_1\cdots \mathsf{Q} v_{n-1} \phi(x,v_0,v_1,\cdots,v_{n-1}) \\ \leftrightarrow \exists y\forall v_1\cdots \mathsf{Q} v_{n-1} \phi_0'(y,v_1,\cdots,v_{n-1}).
        \end{multline*}
        
        Now let us consider the bounded quantifier case. Formally, $\exists x\in a\phi(x)$ is equivalent to $\exists x (x\in a\land \phi(x))$.
        Under the same assumption on $\phi$ and $\phi_0$, $\exists x (x\in a\land \phi(x))$ is equivalent to
        \begin{equation*}
            \exists x \exists v_0 \forall v_0\cdots\mathsf{Q}v_{n-1} [x\in a\land \phi_0(x,v_0,\cdots,v_{n-1})].
        \end{equation*}
        
        The formula $x\in a\land \phi_0(x,v_0,\cdots,v_{n-1})$ is $\Delta_0$, so by exactly the same argument we did for an unbounded $\exists$, we can normalize $\exists x\in a\phi(x)$ to a $\Sigma_n$-formula. Normalizing $\forall x\phi(x)$ or $\forall x\in a\phi(x)$ when $\phi$ is $\Pi_n$ is analogous, so we omit the proof.
        
        \item Suppose that $\phi$ and $\psi$ are $\Sigma_n$-formulas. Showing the normalizability of $\lnot\phi$ is easy, so let us only consider the normalizability of $\phi\land\psi$. The case for $\phi\lor\psi$ is analogous.
        
        Suppose that $\phi$ and $\psi$ are $\Sigma_n$-formulas. Then we have $\Delta_0$-formulas $\phi_0$ and $\psi_0$ such that
        \begin{equation*}
            \phi(\vec{x}) \equiv \exists u_0 \forall u_1 \cdots \mathsf{Q} u_{n-1} \phi_0(u_0,u_1,\cdots, u_{n-1},\vec{x})
        \end{equation*}
        and
        \begin{equation*}
            \psi(\vec{x}) \equiv \exists v_0 \forall v_1 \cdots \mathsf{Q} v_{n-1} \psi_0(v_0,v_1,\cdots, v_{n-1},\vec{x}).
        \end{equation*}
        We may assume that all of $u_i$ and $v_i$ are different variables by substitution, so by basic predicate logic, the following holds:
        \begin{equation*}
            \phi(\vec{x})\land\psi(\vec{x}) \iff \\
            \exists u_0\exists v_0 \forall u_1\forall v_1\cdots \mathsf{Q}u_{n-1}\mathsf{Q}v_{n-1}[\phi_0\land \psi_0].
        \end{equation*}

        Now inductively reduce duplicated quantifiers from the inside to the outside by applying the second clause of \autoref{Lemma: TwoQuantifier-reducing}.

        To illustrate what is happening, let us consider the case $n=2$. In this case, we need to normalize the following formula:
        \begin{equation*}
            \exists u_0\exists v_0 \forall u_1\forall v_1[\phi_0(u_0,u_1)\land \psi_0(v_0,v_1)].
        \end{equation*}
        By Clause 2 of \autoref{Lemma: TwoQuantifier-reducing}, we can find some $\Delta_0$-formula $\chi(u_0,v_0,w_1)$ such that \emph{for every $u_0$ and $v_0$, we have}
        \begin{equation*}
            \forall u_1\forall v_1[\phi_0(u_0,u_1)\land \psi_0(v_0,v_1)] \leftrightarrow \forall w_1 \chi(u_0,v_0,w_1).
        \end{equation*}
        Then applying the equivalent form of the second clause of \autoref{Lemma: TwoQuantifier-reducing}, which is for repeating $\exists$, we can find a $\Delta_0$-formula $\chi'(w_0,w_1)$, we have
        \begin{equation*}
            \exists u_0\exists u_1 \forall w_1 \chi(u_0,v_0,w_1) \leftrightarrow \exists w_0\forall w_1 \chi'(w_0,w_1).
        \end{equation*}

        \item We will prove this claim by induction on $n$.
        $\Sigma_n$-Collection states the following: For every $\Sigma_n$-formula $\phi(x,y,p)$,
        we have for every $a$ and $p$,
        \begin{equation*}
            \forall x\in a \exists y \phi(x,y,p)
            \to \exists b \forall x\in a \exists y\in b \phi(x,y,p).
        \end{equation*}
        Clearly the consequent implies the antecedent, so the equivalence holds.
        Now assume that $\phi$ by a $\Pi_{n-1}$-formula. Then $\exists y\phi(x,y,p)$ becomes a $\Sigma_n$-formula.
        
        Thus, the claim that $\forall x\in a\psi(x)$ is equivalent to a $\Sigma_n$-formula $\psi$ reduces to the claim that $\exists y\in b \chi(y)$ is equivalent to a $\Pi_{n-1}$-formula $\chi$.
        However, by the inductive assumption, $\Sigma_{n-1}$-Collection shows $\exists y\in b \chi(y)$ for a $\Pi_{n-1}$-formula $\chi$ is equivalent to a $\Pi_{n-1}$-formula.
        
        Reducing $\exists x\in a\phi(x)$ for a $\Pi_n$-formula $\phi$ is analogous, so we omit it. \qedhere 
    \end{enumerate}
\end{proof}

In particular, $\Sigma_n$-Collection entails that $\underline{\Sigma}_n$-formulas are provably equivalent to $\Sigma_n$-formulas, and $\underline{\Pi}_n$-formulas are provably equivalent to $\Pi_n$-formulas. In the absence of $\Sigma_n$-Collection, however, we cannot assume that $\underline{\Sigma}_n$-formulas are normalizable to $\Sigma_n$-formulas.

We will frequently define functions and will use them freely in formulas.
The reader might find it challenging to analyze the complexity of formulas with newly introduced functions, so we provide a rigorous treatment of the issue:
\begin{definition}
    Let $\Gamma$ be a class of formulas over the language of a theory $T$.
    Suppose that $\phi(\vec{x},y)$ is a formula in $\Gamma$ and assume that $T$ proves $\forall \vec{x}\exists!y\phi(\vec{x},y)$. Then we say \emph{$T$ proves $\phi$ defines a function}. 

    In practice, we write $F(\vec{x})=y$ instead of $\phi(\vec{x},y)$. We say \emph{$F$ is a $\Gamma$-definable function} if a defining formula $\phi$ is in $\Gamma$. Also, instead of writing $\exists y \phi(\vec{x},y)\land \psi(y)$, we write $\psi(F(\vec{x}))$. 
\end{definition}

Note that if $T$ proves $\phi$ defines a function, $T$ proves $\exists y [\phi(\vec{x},y)\land \psi(y)]$ is equivalent to $\forall y [\phi(\vec{x},y)\to \psi(y)]$.

The following lemma shows composing $\Sigma_1$-definable function with a $\Sigma_1$-formula does not change the complexity:
\begin{lemma} \label{Lemma: Sigma1-composition-complexity}
    Let $F$ be a $\Sigma_1$-definable function and $\psi(x)$ be a $\Sigma_1$-formula or a $\Pi_1$-formula. Then $\psi(F(x))$ is $\Sigma_1$ or $\Pi_1$ respectively.
\end{lemma}
\begin{proof}
    Suppose that $F$ is defined by a $\Sigma_1$-formula $\phi(x,y)$ and $\psi$ is $\Sigma_1$.
    $\psi(F(x))$ is equivalent to $\exists y \phi(x,y)\land \psi(F(x))$, which is normalizable to a $\Sigma_1$-formula by \autoref{Lemma: Formula-normalization}.
    Similarly, if $\psi$ is $\Pi_1$, then $\psi(F(x))$ is equivalent to $\forall y \phi(x,y)\to \psi(y)$, which is normalizable by a $\Pi_1$-formula.
\end{proof}

The following facts are standard; proofs are available in Barwise's \cite{Barwise1975Admissible}:

\begin{lemma}[\cite{Barwise1975Admissible}, Ch. 1, Theorem 4.3] \pushQED{\qed}
    $\KP_1$ proves the $\Sigma$-Reflection principle. That is, for each $\Sigma$-formula $\varphi$, $\KP_1 \vdash \varphi\leftrightarrow\exists a \varphi^a$. Put another way, each $\Sigma$-formula is $\KP_1$-provably equivalent to a $\Sigma_1$-formula.  \qedhere
\end{lemma}

\begin{lemma}[\cite{Barwise1975Admissible}, Ch. 1, Proposition 6.4.]
    \pushQED{\qed}
    Working over $\KP_1$, we have the definition by $\Sigma$-Recursion: Let $G$ be an $(m+2)$-ary $\Sigma$-definable function. Then we can find a $\Sigma$-definable function $F$ such that 
    \begin{equation*}
        F(x_0,\cdots, x_{m-1},y)=G\Big(x_0,\cdots, x_{m-1},y, \big\{\langle z, F(x_0,\cdots, x_{m-1},z)\rangle\mid z \in \TC(y)\big\}\Big),
    \end{equation*}
    where $\TC(x)$ is the transitive closure of $x$.
    
    Furthermore, the formula $z=F(x_0,\cdots,x_m,y)$ is $\Delta_1$.
\end{lemma}
\begin{proof}
    We will not prove the $\Sigma$-Recursion theorem here, and the interested reader may consult \cite{Barwise1975Admissible}. 
    
    Now let us provide the complexity analysis for $z=F(x_0,\cdots,x_m,y)$ when $F$ is recursively defined from an $(m+2)$-ary $\Sigma$-definable function $G$. For notational convenience, let us consider the case $m=1$.

    Observe that $z=F(x,y)$ is equivalent to
    \begin{multline*}
        \exists f {\big[}\text{$f$ is a function of domain $\TC(\{y\})$} \\ \land \forall u\in\TC(\{y\})(f(x,u) = G(x,y,f\upharpoonright \TC(u) ) ){\big]} \land f(x,y)=z.
    \end{multline*}

    The reader can verify that the formula $y=\TC(x)$ is $\KP_1$-provably $\Sigma_1$ (see \cite[\S I.6]{Barwise1975Admissible}), and the statement `$f$ is a function of domain $a$' is $\Delta_0$. Also, if $f$ is a function, then $g=f\upharpoonright a$ is also $\Delta_0$ since
    \begin{equation*}
        g=f\upharpoonright a \iff \forall p\in g(\pi_0(p)\in a\land p\in f) \land \forall p\in f(\pi_0(p)\in a\to p\in g),
    \end{equation*}
    where $\pi_0(p)$ is a projection function satisfying $\pi_0(\langle u,v\rangle)=u$, and the reader can see that both $\pi_0(p)=u$ and $\pi_0(p)\in a$ are $\Delta_0$.
    Hence by applying \autoref{Lemma: Sigma1-composition-complexity}, we have the following:
    \begin{multline} \label{Formula: TC definition}
        \exists f {\Big[}\LaTeXunderbrace{\text{$f$ is a function of domain $\TC(\{y\})$}}_{\Delta_1}  \\ \land \forall \LaTeXunderbrace{u\in\TC(\{y\})}_{\Delta_1}(\LaTeXunderbrace {f(x,u) = G(x,y,f\upharpoonright \TC(u) )}_{\Delta_1} ) {\Big]} \land \LaTeXunderbrace{f(x,y)=z}_{\Delta_0}.
    \end{multline}
    Hence the above formula is $\Sigma_1$. However, we can see that the \eqref{Formula: TC definition} is equivalent to
    \begin{multline*}
        \forall  f {\big[}\text{$f$ is a function of domain $\TC(\{y\})$} \\ \land \forall u\in\TC(\{y\})(f(x,u) = G(x,y,f\upharpoonright \TC(u) ) ){\big]} \to f(x,y)=z,
    \end{multline*}
    and the reader can verify the above formula is $\Pi_1$.
\end{proof}

We will use various types of recursively defined functions, and gauging their complexity is also important. One of the important functions we will use is the function $\alpha\mapsto L_\alpha$, taking an ordinal $\alpha$ and returning the $\alpha$th stage of $L$. By the previous lemma, this function is $\Delta_1$ definable, and the formula $x=L_\alpha$ is equivalent to a $\Delta_1$ formula.

Before finishing this section, let us define partial truth predicates. These play a role analogous to the role that finite Turing jumps $0^{(n)}$ of $0$ play in second-order arithmetic. 
We say that a formula is $\Sigma_n$ if it is of the form $$\exists x_0\forall x_1\cdots \mathsf{Q}x_{n-1}\phi(x_0,\cdots,x_{n-1},\vec{y})$$ for some $\Delta_0$-formula $\phi$.
\begin{lemma}\pushQED{\qed}
   $\KP_1$ proves that for every set $a$ and for every code of a $\Delta_0$-formula $\varphi(x)$, the following are equivalent:
    \begin{itemize}
        \item There is a transitive set $M$ such that $a\in M$ and $M\models \varphi(a)$.
        \item For all transitive $M$ such that $a\in M$, $M\models\varphi(a)$.
    \end{itemize}
    Furthermore, if $\varphi(x)$ is an actual $\Delta_0$-formula and $a$  is a set, $\varphi(a)$ holds iff one of the above holds. \qedhere 
\end{lemma}

The previous lemma allows us to define $\vDash_{\Sigma_0}$ by a $\Delta_1$-expression:
\begin{definition}
    Define $\vDash_{\Sigma_0}$ as follows: $\vDash_{\Sigma_0}\phi(a)$ holds if there is a transitive set $M$ such that $a\in M$ and $M\models \phi(a)$.
\end{definition}

There is no reason to stop defining the partial truth predicate for more complex formulas:
\begin{definition} \label{Definition: Sigma n truth}
    Let us define $\vDash_{\Sigma_n}$ by (meta-)recursion on $n$ as follows: For a $\Pi_n$-formula $\psi(x,y)$,
    $\vDash_{\Sigma_{n+1}}\exists x\psi(x,y)$ holds iff there is $x$ such that $\nvDash_{\Sigma_n} \lnot\psi(x,y)$.
\end{definition}
It is clear that $\vDash_{\Sigma_n}$ is $\Sigma_n$-definable for $n\ge 1$. Furthermore, the above definition works over $\KP_1$ with no problem.
$\vDash_{\Sigma_1}$ behaves like a universal $\Sigma_1$-formula, and is not $\Pi_1$-definable. See \autoref{Lemma: Pi1 boundness for Ord} for a proof of this fact.

\section{Defining `reasonably definable ordinals'}\label{reasonably-definable}

In the context of arithmetic, it is typical to use primitive recursive well-orders as surrogates for ordinals. In this paper, we will work with set theories and we will deal with ordinals in a more general fashion. For instance, we will be working with ordinals that have no primitive recursive presentations. Nevertheless, we will need to restrict our attention to ordinals that are `reasonably definable'. In particular, we will focus on ordinals that are defined by \emph{primitive recursive set relations}. Inspired by \cite{Rathjen1992PrimRecSetFtn}, which defines a set function as a primitive recursive set function if and only if it is provably $\Sigma_1$ over $\KP_1$, we introduce the following definition.

\begin{definition}\label{representation}
    A formula $R(x)$ is a \emph{representation} if $R(x)$ is a $\Sigma_1$-formula such that $\KP_1$ proves 
    \begin{equation*}
        \forall x \forall y\ \Big( \big(R(x)\land R(y)\big)\to x=y\Big).
    \end{equation*}

    An ordinal $\alpha$ is \emph{representable} if there is a representation $R^\bfalpha(x)$ such that $R^\bfalpha$ defines $\alpha$ over $L$. We say that $\alpha$ is representable \emph{via} $R^\bfalpha$.
\end{definition}

For example, $\omega_1^{\mathsf{CK}}$ is representable: It is the least ordinal $\alpha$ such that $L_\alpha\models \KP$, and $\KP_1$ proves there is at most one such ordinal (although $\KP_1$ does not prove that there \emph{is} such an ordinal.) Other examples of representable ordinals are the least gap ordinal\footnote{An ordinal $\alpha$ is a gap ordinal if $L_\alpha\cap\mathcal{P}(\omega) = L_{\alpha+1}\cap\mathcal{P}(\omega)$. It is known that $(\omega, \mathcal{P}(\omega)\cap L_\alpha)$ is a model of full Second-order Arithmetic if $\alpha$ is a gap ordinal.} and the least $\alpha$ such that $L_\alpha\models \mathsf{ZFC}$.

\begin{remark}
    If $R(x)$ is a $\Sigma_1$-formula defining an element, then the following $\Pi_1$ formula also defines the same element:
    \begin{equation} \label{Formula: Pi 1 representation}
        R_*(x) \equiv \forall y \left[ R(y) \to x=y\right].
    \end{equation}
    In particular, every representable ordinal is $\KP_1$-provably $\Delta_1$-definable.
\end{remark}

\begin{remark}
    We often need to reason about representations of ordinals rather than ordinals themselves. We introduce the following conventions to aid with readability: We use boldfaced lowercase Greek letters $\bfalpha$, $\bfbeta$, $\bfgamma$, $\dots$ to denote representations for ordinals, and we handle them as actual ordinals. When we handle $\bfalpha$, $\bfbeta$, $\bfgamma$, $\cdots$ as formulas, we use the notation $R^\bfalpha$, $R^\bfbeta$, $\dots$ to avoid confusion.

    When we write $\varphi(\bfalpha)$, this is really shorthand for $\forall x \big( R^\bfalpha(x) \to \varphi(x)\big)$. Likewise, we write $\Exists(\bfalpha)$ as shorthand for the formula $\exists x R^\bfalpha(x)$. 
\end{remark}

In fact, $\omega_1^{\mathsf{CK}}$ is not only representable but also admissible.\footnote{An ordinal $\alpha$ is admissible if $L_\alpha$ is a model of $\KP$.} Even more, $\KP_1$ proves that it is admissible. The following notion encapsulates this feature of $\omega_1^{\mathsf{CK}}$.
\begin{definition}
    We say that a representation $\bfalpha$ is \emph{$\KP_1$-provably admissible representation} if $\bfalpha$ satisfies
    \begin{equation*}
        \KP_1\vdash \forall \xi [R^\bfalpha(\xi)\to \text{$\xi$ is admissible}].
    \end{equation*}
    The corresponding ordinal $\alpha$ of $\bfalpha$ is called a \emph{$\KP_1$-provably admissible representable ordinal}.
\end{definition}
\emph{Throughout the rest of this paper, we will always assume that $\bfalpha$ is a $\KP_1$-provably admissible representation and that the metatheory satisfies $\Exists(\bfalpha)$.}

We will introduce an $\bfalpha$-analogue of the Kleene normal form theorem. Here the existence of a representable ordinal is analogous to the well-foundedness of a given primitive recursive linear order.

\begin{lemma}\label{Lemma: ordinal-Kleene}
    Let $\exists x A(x)$ be a $\Sigma_1$ sentence. Then there is a representation $\bfgamma$ such that $\KP_1 + (V=L)$ proves the following:
    \begin{itemize}
        \item $\exists x A(x)\leftrightarrow \Exists(\bfgamma)$
        \item $\exists x A(x)\to \forall \xi \big(R^\bfgamma(\xi) \to \exists x\in L_\xi A(x) \big)$
    \end{itemize}
\end{lemma}
\begin{proof}
     Consider $R^\star(\xi)$ given by
    \begin{equation*}
        \text{$\xi$ is an ordinal}\land \exists x\in L_\xi A(x)\land \forall \eta<\xi \; \forall x\in L_\eta \lnot A(x).
    \end{equation*}
    Intuitively, $R^\star(\xi)$ says $\xi$ is the least ordinal such that $L_\xi$ captures a witness of $\exists x A(x)$. 
   
    We need to see that the above formula is $\KP_1$-provably $\Sigma_1$. At first glance, the second conjunct may appear $\Delta_0$, but note that the bounded quantifier is bounded in $L_\xi$, which is defined by $\Sigma$ recursion. That is, the second conjunct contains hidden quantifiers. We can precisely state the second conjunct as follows, which is $\KP_1$-provably equivalent to $\Sigma_1$:
    \begin{equation*}
        \exists a \left[ a = L_\xi \land \left(\exists x\in a A(x)\right)\right].
    \end{equation*}
    The third conjunct of $R^\star(\xi)$ also has the same problem. But the precise statement of the third conjunct is also $\KP_1$-provably equivalent to a $\Sigma_1$:
    \begin{equation*}
        \forall \eta<\xi \exists a [a=L_\eta \land \forall x\in a \lnot A(x)]
    \end{equation*}

Hence if we let $R^\bfgamma$ obtained from $R^\star$ by replacing the two conjuncts to their $\Sigma_1$-equivalent form respectively, then they satisfy the desired properties.
\end{proof}

By relativizing the above lemma to $L_\alpha$ for an admissible $\alpha$, we get the following:
\begin{corollary}\label{Lemma: alpha-Kleene} \pushQED{\qed}
    Let $\exists x A(x)$ be a $\Sigma_1$ sentence. Then there is a representation $\bfgamma$ such that $\KP_1$ proves that for every admissible ordinal $\alpha$, the following holds:
    \begin{itemize}
        \item $\exists x\in L_\alpha A(x)\leftrightarrow \Exists(\bfgamma)^{L_\alpha}$.
        \item $\exists x\in L_\alpha A(x)\to \forall \xi \big(R^\bfgamma(\xi) \to \exists x\in L_\xi A(x) \big)$. 
    \end{itemize}
    Furthermore, the choice of $\bfgamma$ only depends on the formula $\exists x A(x)$. \qedhere 
\end{corollary}

\begin{definition}
    We say that $T$ is $\Sigma^\bfalpha_1$-sound if for all $\varphi \in \Sigma_1$:
    \begin{equation*}
        T\vdash \varphi^{L_\bfalpha} \implies L_\bfalpha \vDash \varphi.
    \end{equation*}
\end{definition}

Now we introduce a notion that resembles the proof-theoretic ordinal of a theory.
\begin{definition}
    Let $\bfalpha$ be $\KP_1$-provably admissible representation. Let $T$ be a $\Sigma^\bfalpha_1$-sound theory extending $\KP_1+\Exists(\bfalpha)$ Then we define: 
    \begin{equation*}
        |T|_{\bfalpha\barrepr} := \sup\{\beta\mid \text{for some $\bfbeta$, $\beta$ is representable via $\bfbeta$ and } T\vdash \Exists(\bfbeta)^{L_\bfalpha}\}.
    \end{equation*}
\end{definition}

Let us start with this property:

\begin{lemma}\label{Lemma: limit-case}
   For all $\beta<|T|_{\bfalpha\barrepr}$, there exists a representable ordinal $\gamma$ represented by $\bfgamma$ such that $\beta<\gamma<|T|_{\bfalpha\barrepr}$ and $T\vdash \Exists(\bfgamma)^{L_\bfalpha}$. Especially, $|T|_{\bfalpha\barrepr}$ is a limit of representable ordinals.
\end{lemma}

\begin{proof}
    Suppose that $\beta<|T|_{\bfalpha\barrepr}$, and without loss of generality, assume that $\beta$ is representable by $\bfbeta$ and $T\vdash\Exists(\bfbeta)^{L_\bfalpha}$. Define a new representation $\boldsymbol{\beta+1}$ by

    \begin{equation*}
        R^{\boldsymbol{\beta+1}}(\xi) \iff \forall\eta\in\xi [R^\bfbeta_*(\eta)\to \xi=\eta\cup\{\eta\}],
    \end{equation*}
    where $R^\bfbeta_*(x)$ is a $\Pi_1$-formula given from $R^\bfbeta$ as presented in \eqref{Formula: Pi 1 representation}.
    
    It is clear that $\KP_1$ proves these two are equivalent, and an ordinal satisfying $R^{\boldsymbol{\beta+1}}(\xi)$ is unique if it exists. Furthermore, $T$ proves $\exists \xi R^{\boldsymbol{\beta+1}}_\Sigma(\xi)$, so $\beta+1<|T|_{\bfalpha\barrepr}$. Thus $|T|_{\bfalpha\barrepr}$ is a limit of representable ordinals.
\end{proof}

\begin{remark}
    Let us remark that we could prove that $|T|_{\bfalpha\barrepr}$ is closed under addition and multiplication, but this would require a more sophisticated argument.
\end{remark}

Finally, it turns out that the new ordinal coincides with a notion that we have already discussed.
\begin{lemma}\label{Lemma: equality-invariants}
    Let $\bfalpha$ be a $\KP_1$-provably admissible representation and $T$ be a $\Sigma^\bfalpha_1$-sound theory extending $\KP_1 +  \Exists(\bfalpha)$. Then $|T|_{\bfalpha\barrepr}=|T|_{\Sigma^\bfalpha_1}$.
\end{lemma}
\begin{proof}
    We prove $|T|_{\bfalpha\barrepr}\le |T|_{\Sigma^\bfalpha_1}$ first.
    Suppose that $\beta$ is an ordinal and $\beta<|T|_{\bfalpha\barrepr}$. 
    We claim that $\beta<|T|_{\Sigma^\bfalpha_1}$.
    
    Since $\beta<|T|_{\bfalpha\barrepr}$, \autoref{Lemma: limit-case} entails that there is a representable $\gamma\in(\beta,|T|_{\bfalpha\barrepr})$ such that $T$ proves $\Exists(\bfgamma)^{L_\bfalpha}$.
    Note that $\Exists(\bfgamma)^{L_\bfalpha}$ is a $\Sigma^\bfalpha_1$ formula, so by the definition of $|T|_{\Sigma^\bfalpha_1}$, we have $L_{|T|_{\Sigma^\bfalpha_1}}\models \exists \xi R^\bfgamma(\xi)$. That is, we have the following:
    \begin{equation*}\tag{$A$}
        L \models \exists \xi<|T|_{\Sigma^\bfalpha_1}\ (R^\bfgamma_\Sigma(\xi))^{L_{|T|_{\Sigma^\bfalpha_1}}}.
    \end{equation*}
    By upward absoluteness, ($A$) implies that there exists $\xi<|T|_{\Sigma^\bfalpha_1}$ such that $ L_\alpha\models R^\bfgamma(\xi)$ and so $\gamma<|T|_{\Sigma^\bfalpha_1}$.
    Hence $\beta < \gamma < |T|_{\Sigma_1^\bfalpha}$.

    Now let us prove $|T|_{\Sigma^\bfalpha_1}\le |T|_{\bfalpha\barrepr}$.
    Suppose that $\gamma < |T|_{\Sigma^\bfalpha_1}$. Then there is a $\Delta_0$-formula $A(x)$ such that 
    \begin{equation*}
        T\vdash \exists x\in L_\bfalpha A(x) \quad\text{but}\quad L_\gamma\nvDash \exists x A(x).
    \end{equation*}

    By \autoref{Lemma: ordinal-Kleene}, we can find a representation $\bfdelta$ such that $\KP_1 + (V=L)$ proves the following:
    \begin{enumerate}
        \item[(C1)] $\exists x A(x)\leftrightarrow \Exists(\bfdelta)$.
        \item[(C2)] $\exists x A(x)\to \exists x\in L_\bfdelta A(x)$.
    \end{enumerate}

    Then we have $T\vdash \Exists(\bfdelta)^{L_\bfalpha}$: Reasoning over $T$, we have $\exists x\in L_\bfalpha A(x)$ by  assumption. Since $L_\bfalpha$ is a model of $\KP_1 + (V=L)$, we have
    \begin{equation*}
        L_\bfalpha \vDash \exists x A(x)\leftrightarrow \Exists(\bfdelta),
    \end{equation*}
    so we get $\Exists(\bfdelta)^{L_\bfalpha}$. 
    Reasoning externally, since $T$ is $\Sigma_1^\bfalpha$-sound and $\Exists(\bfdelta)^{L_\bfalpha}$ is $\Sigma_1^\bfalpha$, we have that $\bfdelta$ exists over the metatheory.
    Furthermore, we have $\bfdelta < |T|_{\bfalpha\barrepr}$.
    
    Now let us prove $\gamma < \bfdelta$: 
    We know that $\Exists(\bfdelta)^{L_\bfalpha}$ is true.
    Since $L_\bfalpha$ is a model of $\KP_1 + (V=L)$, we have
    \begin{equation*}
        L_\bfalpha \vDash \exists x A(x)\to \exists x\in L_\bfdelta A(x).
    \end{equation*}
    This implies $\exists x\in L_\bfdelta A(x)$ is true. If $\gamma\ge \bfdelta$, then $L_\gamma\vDash \exists x A(x)$ by upward absoluteness, a contradiction. Hence $\gamma < \bfdelta$.
\end{proof}

\section{Kreiselian relations}\label{kreisel-relations}

One can interpret $\bfalpha$ ordinal analysis as a means of partitioning theories: Identify two theories if they have the same $\bfalpha$ proof-theoretic ordinal. In this section, we demonstrate the robustness of this partition. Indeed, the $\bfalpha$ ordinal analysis partition is the finest partition satisfying some very natural properties. To explicate these properties, we first introduce some definitions.

\begin{definition}
    For theories $T$ and $U$ and any syntactic complexity class $\Gamma$, $T\subseteq_\Gamma U$ if, for every $\varphi \in \Gamma$ such that $T\vdash \varphi$, $U\vdash \varphi$. $T\equiv_\Gamma U$ if both $T\subseteq_\Gamma U$ and $U\subseteq_\Gamma T.$
\end{definition}

\begin{definition}\label{a-Kreis}
    An equivalence relation $\equiv$ on $\Sigma^\bfalpha_1$-sound and $\Pi^\bfalpha_1$-definable extensions of $\mathsf{KP}_1+\Exists(\bfalpha)$ is \emph{$\bfalpha$-Kreiselian} if both:
    \begin{enumerate}
        \item $T\equiv_{\Sigma^\bfalpha_1}U$ implies $T\equiv U$.
        \item $T\equiv T+\theta$ for every $\Pi^\bfalpha_1$ sentence $\theta$.
    \end{enumerate}
\end{definition}

The motivating example of an $\bfalpha$-Kreiselian partition is the aforementioned $\bfalpha$ ordinal analysis partition.
\begin{definition}
    For representable $\bfalpha$, the \emph{$\bfalpha$-ordinal analysis partition} is the partition of  $\Pi^\bfalpha_1$-definable $\Sigma^\bfalpha_1$-sound extensions of $\KP_1 +\Exists(\bfalpha)$ such that the cell of $T$ is $\{ U : |T|_{\Sigma^\bfalpha_1}=|U|_{\Sigma^\bfalpha_1} \}$.
\end{definition}

\begin{proposition}
    The $\bfalpha$-ordinal analysis partition is $\bfalpha$-Kreiselian.
\end{proposition}
\begin{proof}
    Clearly $T\equiv_{\Sigma^\bfalpha_1} U$ implies $|T|_{\Sigma^\bfalpha_1}=|U|_{\Sigma^\bfalpha_1}$ since $|T|_{\Sigma^\bfalpha_1}$ is determined by the $\Sigma^\bfalpha_1$-consequences of $T$.
    The remaining condition follows from Theorem 1.2.5 of \cite{Pohlers1998Subsystems}.
\end{proof}

The next theorem says the $\bfalpha$-Kreiselian relation induced by $|T|_{\Sigma^\bfalpha_1}$ is the finest $\bfalpha$-Kreiselian relation:

\begin{theorem}\label{Theorem: FirstEquivalence}
    Let $\alpha$ be a $\KP_1$-provably admissible representation $\bfalpha$ and $\equiv$ be an $\boldsymbol{\alpha}$-Kreiselian relation. Let $T$ and $U$ are $\Pi^\bfalpha_1$-definable $\Sigma^\bfalpha_1$-sound extensions of $\KP_1 +\Exists(\bfalpha)$ such that $|T|_{\Sigma^\bfalpha_1}=|U|_{\Sigma^\bfalpha_1}$. Then $T\equiv U$.
\end{theorem}

\begin{proof}
    Let $A(x)$ be a $\Delta_0$-formula and assume that $T\vdash \exists x\in L_\alpha A(x)$. By \autoref{Lemma: alpha-Kleene}, we can find a representation $\bfgamma$ such that
    \begin{equation*}
        \KP_1+\Exists(\bfalpha)\vdash \exists x\in L_\bfalpha\ A(x)\leftrightarrow \Exists(\bfgamma)^{L_\bfalpha}.
    \end{equation*}
    
    Then we have $\bfgamma<|T|_{\Sigma^\bfalpha_1}=|U|_{\Sigma^\bfalpha_1}$, and $L_{|U|_{\Sigma^\bfalpha_1}}\models \Exists(\bfgamma)$.
    By \autoref{Lemma: limit-case}, $|U|_{\Sigma^\bfalpha_1}$ is a limit of representable ordinals that $U$-provably exist, 
    so we can find some representable ordinal $\bfdelta<|U|_{\Sigma^\bfalpha_1}$ such that $\bfgamma< \bfdelta$, $U\vdash\Exists(\bfdelta)^{L_\bfalpha}$, and $L_\bfdelta\models\Exists(\bfgamma)$.%
    \footnote{We know that $L_{|U|_{\Sigma^\bfalpha_1}}\models \Exists(\bfgamma)$, and $|U|_{\Sigma^\bfalpha_1}$ is a limit of representable ordinals that provably exist over $U$. 
    Since the claim $\Exists(\bfgamma)\equiv \exists\eta R^\bfgamma(\eta)$ is $\Sigma_1$, we can find a desired representable $\bfdelta$.}
    
    Now consider the sentence $\Comp(\bfgamma,\bfdelta)$ given by
    \begin{equation*}
        \Comp(\bfgamma,\bfdelta)\equiv \forall\xi[R^\bfdelta(\xi)\to L_\xi \models \Exists(\bfgamma)].
    \end{equation*}
    Roughly, $\Comp(\bfgamma,\bfdelta)$ says that if $\bfdelta$ exists then $\bfgamma$ exists and $\bfgamma < \bfdelta$. So, in some sense, $\Comp(\bfgamma,\bfdelta)$ says that $\bfgamma<\bfdelta$. Also, note that $\Comp(\bfgamma,\bfdelta)$ is a true $\Pi^\bfalpha_1$-sentence as long as $\Exists(\bfgamma)$ and $\Exists(\bfdelta)$ are true and $\bfgamma<\bfdelta$.
    
    Now we can see that $U$ + $\Comp(\bfgamma,\bfdelta)$ proves $L_\bfdelta\models \Exists(\bfgamma)^{L_\bfalpha}$: 
    Working inside $U$ + $\Comp(\bfgamma,\bfdelta)$, we have $\Exists(\bfdelta)^{L_\bfalpha}$, 
    which implies $\bfdelta < \bfalpha$.
    Then by $\Comp(\bfgamma,\bfdelta)$, we have $L_\bfdelta\vDash \Exists(\bfgamma)$.
    Now by upward absoluteness, we have $L_\bfalpha\vDash \Exists(\bfgamma)$.

    Returning back to reasoning externally, we have
    $U+\Comp(\bfgamma,\bfdelta)\vdash \Exists(\bfgamma)^{L_\bfalpha}$. Hence $U+\Comp(\bfgamma,\bfdelta)$ also proves $\exists x A(x)$.
    To rephrase the conclusion, we proved the following:

    \begin{itemize}
        \item[($\star$)]  For every $\Delta_0$-formula $A(x)$, if $T\vdash \exists x\in L_\alpha A(x)$, then we can find two representable ordinals $\bfgamma$ and $\bfdelta$ such that $U\vdash \Exists(\bfdelta)$ and $U + \Comp(\bfgamma,\bfdelta)\vdash \exists x\in L_\bfalpha A(x)$.
    \end{itemize}

    Define $T_{\Sigma^\bfalpha_1}$ as the set of $\Sigma^\bfalpha_1$ theorems of $T$ and $U_{\Sigma^\bfalpha_1}$ as the set of $\Sigma^\bfalpha_1$ theorems of $U$. Now we define an extension $\widehat{U}$ of the theory $U$.  Precisely, $\varphi \in \widehat{U}$ if and only if one of the following holds:
    \begin{enumerate}
        \item $\varphi \in U$.
        
        \item $\varphi$ has the form $\Comp(\bfgamma,\bfdelta)$, where $\bfgamma$ and $\bfdelta$ are representations and 
        \begin{equation*}
            \exists\psi\in T_{\Sigma^\bfalpha_1} \left[ \KP_1+\Exists(\bfalpha) \vdash \psi\leftrightarrow \Exists(\bfgamma),\, 
            U\vdash \Exists(\bfdelta),\, L_\bfdelta\models \Exists(\bfgamma)\right].
        \end{equation*}
        
        \item $\varphi$ has the form $\Comp(\bfgamma,\bfdelta)$, where $\bfgamma$ and $\bfdelta$ are representations and
        \begin{equation*}
            \exists\psi\in U_{\Sigma^\bfalpha_1} \left[ \KP_1+\Exists(\bfalpha) \vdash \psi\leftrightarrow \Exists(\bfgamma),\, T\vdash \Exists(\bfdelta),\, L_\bfdelta\models \Exists(\bfgamma)\right].
        \end{equation*}
    \end{enumerate}
    We define $\widehat{T}$ analogously except that we replace clause (1) with $\varphi \in T$.
    By ($\star$), $\widehat{T}$ and $\widehat{U}$ prove all of the $\Sigma^\bfalpha_1$-theorems of $T$ and $U$.

    Furthermore, $\widehat{T}$ and $\widehat{U}$ are also $\Pi^\bfalpha_1$-definable since both of $T$ and $U$ are $\Pi^\bfalpha_1$-definable and ``$L_\bfdelta\models \Exists(\bfgamma)$'' is expressible by the following $\Pi^\bfalpha_1$ formula:
    \begin{equation*}
        \left[\forall \eta\ \left(R^\bfdelta (\eta)\to L_\eta\models \Exists(\bfgamma)\right)\right]^{L_\bfalpha}.
    \end{equation*}

    We extended $T$ to $\widehat{T}$ by adding true $\Pi^\bfalpha_1$-sentences; hence $T\equiv \widehat{T}$ by condition 2 from Definition \ref{a-Kreis}. Similarly, $U\equiv \widehat{U}$.
    
    Now we claim that $\widehat{T}\equiv_{\Sigma^\bfalpha_1}\widehat{U}$, which will finalize our proof since it implies $\widehat{T}\equiv \widehat{U}$, so we have $T\equiv \widehat{T} \equiv \widehat{U} \equiv U$.
    Indeed, it suffices to show that $\widehat{T}\subseteq_{\Sigma^\bfalpha_1} \widehat{U}$ since the converse containment is symmetric.
    
    Suppose that $\widehat{T}\vdash\theta$ for some $\Sigma^\bfalpha_1$-sentence $\theta$. Then $T+\sigma\vdash \theta$ for some true $\Pi^\bfalpha_1$-sentence $\sigma$, which is a conjunction of sentences of the form $\Comp(\bfgamma,\bfdelta)$ we added.
    Thus we have $T\vdash \sigma\to\theta$, and $\sigma\to\theta$ is $\Sigma^\bfalpha_1$.
    However, we defined $\widehat{U}$ so that it proves all $\Sigma^\bfalpha_1$-consequences of $T$, so $\widehat{U}\vdash\sigma\to \theta$. Since $\sigma\in \widehat{U}$, we have $\widehat{U}\vdash\theta$.
\end{proof}

What are the $\bfalpha$-Kreselian equivalence relations, other than the ones induced by generalized ordinal analysis? In the rest of this section, we will give an exact answer to this question. This answer has another benefit. Theorem \ref{Theorem: FirstEquivalence} says the $\alpha$-ordinal analysis partition is a maximally fine $\boldsymbol{\alpha}$-Kreiselian equivalence relation. But we will show that it is indeed the unique finest $\alpha$-Kreseilian equivalence relation.

\begin{definition}
    $\Ord_{\Sigma^\bfalpha_1}$ is the set $$\big\{ |T|_{\Sigma^\bfalpha_1} : \text{ $T$ is a $\Pi^\bfalpha_1$-definable $\Sigma^\bfalpha_1$-sound extension of $\KP_1 +\Exists(\bfalpha)$} \big\}.$$
\end{definition} 

\begin{definition}
    Given a partition $\sim$ of $\Ord_{\Sigma^\bfalpha_1}$, the \emph{$\sim$-induced partition} $\equiv$ on theories is defined as follows:
$$T\equiv U\text{ if and only if }|T|_{\Sigma^\bfalpha_1}\sim|U|_{\Sigma^\bfalpha_1}.$$ We say that $\equiv$ is \emph{$\Ord_{\Sigma^\bfalpha_1}$ induced} if $\equiv$ is $\sim$-induced for some partition $\sim$ of $\Ord_{\Sigma^\bfalpha_1}$.
\end{definition}

\begin{remark}
    Since every equivalence relation coarsens the identity relation, every $\Ord_{\Sigma^\bfalpha_1}$ induced relation coarsens the $\bfalpha$-ordinal analysis partition.
\end{remark}

\begin{theorem}\label{main-kreisel}
The $\bfalpha$-Kreiselian equivalence relations are exactly the $\Ord_{\Sigma^\bfalpha_1}$ induced partitions.
\end{theorem}

\begin{proof}
First, we need to see that every $\Ord_{\Sigma^\bfalpha_1}$ induced partition is $\boldsymbol{\alpha}$-Kreiselian. Let $\equiv$ be $\Ord_{\Sigma^\bfalpha_1}$ induced.

\begin{enumerate}
    \item Suppose that $T\equiv_{\Sigma^\bfalpha_1}U$. Then $|T|_{\Sigma^\bfalpha_1}=|U|_{\Sigma^\bfalpha_1}$, since the $\alpha$-ordinal analysis partition is $\boldsymbol{\alpha}$-Kreiselian. Then $T \equiv U$ since $\equiv$ coarsens the $\alpha$-ordinal analysis partition.
    \item Let $\varphi$ be true $\Pi^\bfalpha_1$. Then $|T|_{\Sigma^\bfalpha_1}=|T+\varphi|_{\Sigma^\bfalpha_1}$. Then $T \equiv T+\varphi$ since $\equiv$ coarsens the $\bfalpha$-ordinal analysis partition.
\end{enumerate}

Next we must show that every $\boldsymbol{\alpha}$-Kreiselian partition is $\Ord_{\Sigma^\bfalpha_1}$ induced. Let $\equiv$ be $\boldsymbol{\alpha}$-Kreiselian. We define a partition $\sim$ on $\Ord_{\Sigma^\bfalpha_1}$ as follows: For $\beta,\gamma\in \Ord_{\Sigma^\bfalpha_1}$, $\beta\sim\gamma$ if and only if, for some $\Pi^\bfalpha_1$-definable $\Sigma^\bfalpha_1$-sound extensions $V$ and $W$ of $\KP_1 +\Exists(\bfalpha)$, the following three claims hold:
\begin{enumerate}
    \item $|V|_{\Sigma^\bfalpha_1}=\beta$
    \item $|W|_{\Sigma^\bfalpha_1}=\gamma.$
    \item $V\equiv W$.
\end{enumerate}

It is trivial to establish that $\sim$ is an equivalence relation on $\Ord_{\Sigma^\bfalpha_1}$. So we need only see that $\sim$ induces $\equiv$. To this end, we define:
$$T\equiv_\star U := |T|_{\Sigma^\bfalpha_1}\sim |U|_{\Sigma^\bfalpha_1}.$$

Since $\equiv_\star$ just is the $\sim$-induced equivalence relation, it suffices to show that $T\equiv U$ if and only if $T\equiv_\star U$. 

Suppose that $T\equiv U$.  So by the definition of $\sim$, we have $|T|_{\Sigma^\bfalpha_1}\sim |U|_{\Sigma^\bfalpha_1}$. But then $T\equiv_\star U$.

Conversely, suppose that $T\equiv_\star U$, i.e., that $|T|_{\Sigma^\bfalpha_1}\sim |U|_{\Sigma^\bfalpha_1}$. By definition of $\sim$, we infer that for some $T'$ and $U'$:
\begin{enumerate}
    \item $|T'|_{\Sigma^\bfalpha_1}=|T|_{\Sigma^\bfalpha_1}$.
    \item $|U'|_{\Sigma^\bfalpha_1}=|U|_{\Sigma^\bfalpha_1}.$
    \item $T' \equiv U'$.
\end{enumerate}

Applying Theorem \ref{Theorem: FirstEquivalence} to (1) yields $T\equiv T'$. Likewise, applying Theorem \ref{Theorem: FirstEquivalence} to (2) yields $U\equiv U'$.

Combining these observations with (3) yields $T\equiv T' \equiv U' \equiv U$, whence $T\equiv U$.
\end{proof}

\begin{corollary}
The $\bfalpha$-ordinal analysis partition properly refines every other $\boldsymbol{\alpha}$-Kreiselian partition.
\end{corollary}

\begin{corollary}
    The $\bfalpha$-ordinal analysis partition is the unique finest $\bfalpha$-Kreiselian partition.
\end{corollary}

Recall that $\Pi^1_1$-statements are precisely $\Sigma_1^{\omega_1^{CK}}$-sentences. So for $\alpha=\omega_1^{\mathsf{CK}}$, i.e., the classic ordinal analysis, we can say more. The following theorem follows is probably ``folklore.''

\begin{proposition}\label{exact-epsilon}
    The proof-theoretic ordinals of the $\Pi^1_1$-sound $\Sigma^1_1$-definable extensions of $\mathsf{ACA}_0$ are exactly the recursive $\varepsilon$-numbers.
\end{proposition}

Whence it follows that:
\begin{theorem}
The $\omega_1^{\mathsf{CK}}$-Kreiselian equivalence partitions are exactly those that are induced by partitions of the recursive $\varepsilon$-numbers.\footnote{Theorem \ref{main-kreisel} was proved by considering extensions of $\KP_1 +\Exists(\bfalpha)$. This strong theory is required only since we are uniformly proving results about $\bfalpha$ ordinal analysis for all $\bfalpha$. One can easily check that the special case of Theorem \ref{main-kreisel} needed for classic $\Pi^1_1$ ordinal analysis goes through when we consider extensions of $\mathsf{ACA}_0$.}
\end{theorem}

\section{A connection with $\alpha$-recursion}\label{alpha-recursion}
Another well-established method for measuring the proof-theoretic strength of $T$ is by characterizing the $T$-provably total recursive functions. This yields a computational perspective on ordinal analysis (see $\Pi^0_2$ proof-theoretic ordinals in \cite{beklemishev2005reflection}). In this paper, we have been investigating ``higher'' analogues of the usual $\Pi^1_1$ ordinal analysis. Given the connection between ordinal analysis and provably total recursive functions, it is natural to wonder whether there is a significant connection between these higher analogues of ordinal analysis and generalized computability, e.g., the $T$-provably total $\alpha$-recursive functions.

There are many ways of defining the notion of an $\alpha$-recursive functions. For instance, Pohlers defines  the $\omega_1^{\mathsf{CK}}$-recursive functions as functions $$f\colon L_{\omega_1^{\mathsf{CK}}}\to L_{\omega_1^{\mathsf{CK}}}$$ whose graph is $\Sigma_1^{\omega_1^{\mathsf{CK}}}$-definable. Using this notion, one can easily establish a connection between $|T|_{\Sigma^\bfalpha_1}$ and the provably total $\alpha$-recursive functions of $T$. However, the definition Pohlers uses emphasizes \emph{definability} rather than \emph{computation}. Since our definition of $|T|_{\Sigma^\bfalpha_1}$ straightforwardly concerns syntactic complexity, it is neither difficult nor enlightening to see its connection with the definability-centric notion Pohlers uses. 

Perhaps the most famous definition of $\alpha$-recursion was provided by Sacks in \cite{sacks2017higher}, by using the $\Sigma_1$-truth predicate over $L$. So, again, this is a definability-oriented notion. Although Sacks' formulation is simple and standard, adapting his formulation to formulate a higher analogue of ``gauging the strength of theories via provably total recursive function'' is not too enlightening in the present context. 

Hence, in this section, we will focus instead on \emph{computation-centric} definitions of $\alpha$-recursion. By connecting these with $|T|_{\Sigma^\bfalpha_1}$, we can develop a computational perspective on higher ordinal analysis. In particular, we will turn our attention to operational formulations of $\alpha$-recursion. Hamkins and Lewis \cite{HamkinsLewis2000}, Koepke \cite{Koepke2005TuringforOrdinals}, and others have defined Turing-machine-like formulations of ordinal computation.
Koepke's $\alpha$-machine coincides with $\alpha$-recursion for ordinals \cite{koepke2009ordinal}. Unfortunately, for these operational notions, the only valid inputs to computations are \emph{ordinals}. By contrast, Sacks' notion concerns arbitrary sets (in $L$).

Recently, Passmann \cite{Passmann2022CZFintuitionistic} introduced a set-friendly operational computability notion---\emph{the set register machine}--- to define his set realizability model. He formulated this model of generalized computation in $\ZFC$ with the powerset axiom. In fact, it is possible to formulate this model of computation over $\KP_1 + (V=L)$ without the powerset axiom. We now provide the definition of set register machine for completeness (following Definitions 3.1 and 3.2 in \cite{Passmann2022CZFintuitionistic}).

Note that in this section, $<_L$ is the standard $\Sigma_1$-definable global well-order over $L$. Note that this definition works in the setting $\KP_1$.
\begin{definition}[$\KP_1 + (V=L)$]
    A \emph{set register program} or \emph{set register machine} (abbreviated \emph{SRM}) $p$ is a finite sequence of instructions, where an \emph{instruction} is an instance of one of the following commands (in what follows, $\mathtt{R}_i$ designates the $i$th bit of content of the $i$th register).
    \begin{enumerate}
        \item `$\mathtt{R}_i:=\varnothing$' Replace the contents of the $i$th register with the empty set.
        \item `$\mathtt{ADD}(i,j)$' Replace the content of the $j$th register with $\mathtt{R}_j\cup\{\mathtt{R}_i\}$.
        \item `$\mathtt{COPY}(i,j)$' Replace the content of the $j$th register with $\mathtt{R}_i$.
        \item `$\mathtt{TAKE}(i,j)$' Replace the content of the $j$th register with the $<_L$-least set contained in $\mathtt{R}_i$ if $\mathtt{R}_i\neq\varnothing$. Otherwise, do nothing.
        \item `$\mathtt{REMOVE}(i,j)$' Replace the content of the $j$th register with the set $\mathtt{R}_i\setminus \{\mathtt{R}_j\}$.
        \item `\texttt{IF $\mathtt{R}_i=\varnothing$ THEN GO TO $k$}' Check whether $\mathtt{R}_i=\varnothing$. If so, move to program line $k$. If not, move to the next line.
        \item `\texttt{IF $\mathtt{R}_i\in \mathtt{R}_j$ THEN GO TO $k$}' Check whether $\mathtt{R}_i\in \mathtt{R}_j$. If so, move to program line $k$. If not, move to the next line.
    \end{enumerate}

    For an SRM $p$---where $k\in\mathbb{N}$ is the largest register index appearing in $p$---we say that a sequence $c=\langle l,r_0,\cdots,r_k\rangle$ is a \emph{configuration} if $l$ is a natural number. Intuitively, $l$ is the active program line, and $r_i$ is the content of the $i$th register machine.
    For a configuration $c$, its \emph{successor configuration} $c^+=\langle l^+, r_0^+,\cdots, r_k^+\rangle$ is the following:
    \begin{enumerate}
        \item If $p_l$ is `$\mathtt{R}_i:=\varnothing$', then $r_i^+=\varnothing$, $r_j^+=r_j$ for $j\neq i$, $l^+=l+1$.
        \item If $p_l$ is `$\mathtt{ADD}(i,j)$', then $r_j^+ = r_j\cup\{r_i\}$, $r_m^+=r_m$ if $m\neq j$, $l^+=l+1$.
        \item If $p_l$ is `$\mathtt{COPY}(i,j)$', then $r_j^+ = r_i$, $r_m^+ = r_m$ if $m\neq j$, $l^+ = l+1$.
        \item If $p_l$ is `$\mathtt{TAKE}(i,j)$', then
        \begin{equation*}
            r_j^+ = \begin{cases}
                \text{the $<_L$-minimal element of $r_i$} & \text{if } r_i\neq\varnothing, \\
                r_j & \text{otherwise},
            \end{cases}
        \end{equation*}
        $r_m^+ = r_m$ if $m\neq j$, $l^+=l+1$.
        \item If $p_l$ is `$\mathtt{REMOVE}(i,j)$', then $r_j^+ = r_j\setminus\{r_i\}$, $r_m^+ = r_m$ for $m\neq j$, $l^+ = l+1$.
        \item If $p_l$ is `\texttt{IF $\mathtt{R}_i=\varnothing$ THEN GO TO $k$}', then $r_m^+ = r_m$ for all $m\le k$, and 
        \begin{equation*}
            l^+ = \begin{cases}
                m & \text{if }r_i=\varnothing, \\
                l+1 & \text{otherwise}.
            \end{cases}
        \end{equation*}
        \item If $p_l$ is `\texttt{IF $\mathtt{R}_i\in \mathtt{R}_j$ THEN GO TO $k$}', then $r_m^+ = r_m$ for all $m\le k$, and
        \begin{equation*}
            l^+ = \begin{cases}
                m & \text{if }r_i\in r_j, \\
                l+1 & \text{otherwise}.
            \end{cases}
        \end{equation*}
    \end{enumerate}

    A \emph{computation of $p$ with input $x_0,\cdots, x_j$} is a sequence $\langle d_\xi \mid \xi < \alpha+1 \rangle$ with some successor ordinal length consisting of the configurations of $p$ such that
    \begin{enumerate}
        \item $d_0 = \langle 1,x_0,\cdots, x_j,\varnothing,\cdots,\varnothing\rangle$;
        \item If $\beta<\alpha$, then $d_{\beta+1}=d_\beta^+$;
        \item If $\beta<\alpha$ is a limit and $d_\gamma = \langle l_\gamma, r_\gamma^0,r_\gamma^1,\cdots, r_\gamma^{k-1}\rangle$ for $\gamma<\beta$, then $l_\beta = \liminf_{\gamma<\beta}l_\gamma$,  $r^i_\beta = \liminf_{\gamma<\beta} r^i_\gamma$ for $i<k$, and
        \begin{equation*}
            d_\beta = \langle l_\beta,r_\beta^0,r_\beta^1,\cdots, r_\beta^{k-1} \rangle
        \end{equation*}
        where the limit inferior of sets are computed under the order $<_L$;\footnote{That is, for a sequence of sets $\langle r_\gamma\mid\gamma<\beta\rangle$, $r=\liminf_{\gamma<\beta} r_\gamma$ if and only if $r=\sup\{s\mid s\le r_\gamma\text{ eventually}\}$, where $\sup$ is computed under $<_L$.}
        \item $d_\alpha^+$ is undefined.
    \end{enumerate}
\end{definition}

It can be verified that most arguments in Section 3 of \cite{Passmann2022CZFintuitionistic} are formalizable in $\KP_1 + (V=L)$ as long as the proof does not use the powerset operation. In particular, we can prove the following:
\begin{lemma}[$\KP_1+(V=L)$, Lemma 3.18 of \cite{Passmann2022CZFintuitionistic}]
    \pushQED{\qed}
    There is an SRM $\mathsf{Tr}_{\Delta_0}$, which takes a code for a $\Delta_0$-formula $\phi$ and a parameter $a$, such that $\mathsf{Tr}_{\Delta_0}(\ulcorner\phi\urcorner,a)=1$ if $\vDash_{\Sigma_0}\phi(a)$, and $\mathsf{Tr}_{\Delta_0}(\ulcorner\phi\urcorner,a)=0$ if $\vDash_{\Sigma_0}\lnot\phi(a)$. \qedhere 
\end{lemma}

It turns out that the halting problem for Passman's set register machine is  $\Sigma_1$-complete over $L$:
\begin{lemma}[$\KP_1+(V=L)$] \label{Lemma: Sigma-L-1 equals to L-TM-halts}
Let $p$ be an SRM. Then the statement `The program $p$ with input $x$ halts' is $\Sigma_1$. Conversely, every $\Sigma_1$ formula $\varphi(x)$ is equivalent to a statement of the form `An SRM $p$ with an input $x$ halts.'

Furthermore, the statement `The program $p$ with an input $x$ halts' is equivalent to a statement of the form $\exists d \; \mathtt{T}(p,x,d)$ for some $\KP_1$-provably $\Delta_1$ formula $\mathtt{T}$, and satisfies if $d\in L_\beta$, then $x$ and the output of $p$ with input $x$ are in $L_\beta$.
\end{lemma}
\begin{proof}[Sketch of the proof]
    Formally, `$M$ halts with an input $x$' is equivalent to the existence of a computation $d$ of $p$ with input $x$, and we can see that $d$ being a computation is $\Delta_1$. If we take $\mathtt{T}(p,x,d)$ to be
    \begin{equation*}
        \mathtt{T}(p,x,d):= \text{$d$ is a computation of $p$ with input $x$,}
    \end{equation*}
    and if $d\in L_\beta$, then we have that the output of $p$ under $x$ is in $L_\beta$ since $d$ contains the output of the computation.
    
    Conversely, suppose that $\exists y \psi(x,y)$ is a $\Sigma_1$-formula, where $\psi$ is a bounded formula. Now consider an SRM $p$ trying to find $y$ such that $\mathsf{Tr}_{\Delta_0}(\ulcorner\psi\urcorner,x,y)$ holds via unbounded search. Then `$p$ halts with input $x$' is equivalent to $\exists y \psi(x,y)$.
\end{proof}

Also, we can see that the register machine computable class functions coincide with $\Sigma_1$-definable class functions:
\begin{lemma}[$\KP_1+(V=L)$]
    A class partial function $F\colon L\to L$ is $\Sigma_1$-definable (in the sense that $x\in \operatorname{dom} F$ and $y=F(x)$ are $\Sigma_1$) if and only if there is an SRM $p$ and a parameter $a$ such that $x\in\operatorname{dom} F$ iff $p(x,a)$ halts and $F(x)=p(x,a)$ for all $x$ in the domain.
    
    In fact, we can effectively determine $p$ from the $\Sigma_1$-definition of $F$.
\end{lemma}
\begin{proof}[Sketch of the proof]
    In one direction, suppose that the graph of $F\colon L\to L$ is defined by a formula $\exists z \phi_0(x,y,z,a)$ for some $\Delta_0$ formula $\phi_0$ and a parameter $a$.
    Then consider the following SRM $p(x,a)$: Enumerate all sets under the $<_L$-order. If the set we enumerate is not an ordered pair, skip this set. If the set we enumerate is of the form $\langle y,z\rangle$, compute $\Tr_{\Delta_0}(\ulcorner \phi_0\urcorner, x,y,z,a)$.
    If $\Tr_{\Delta_0}(\ulcorner \phi_0\urcorner, x,y,z,a)=1$, $p(x,a)$ returns $y$. Otherwise, proceed with the computation.
    Then by definition of $p$, $p(x,a)$ halts iff $x\in \operatorname{dom} F$, and $p(x,a)=y$ iff $F(x)=y$.

    Conversely, suppose that $F$ is a class function defined by an SRM $p(x,a)$ with parameter $a$. We can see that there is a $\Delta_1$-definable function $\mathtt{U}$, which reads off a computation $d$ and extracts its output. Then $F(x)=y$ if and only if $\exists d\,\mathtt{T}(p,\langle x,a\rangle,d)\land y=\mathtt{U}(d)$, which is $\Sigma_1$.
\end{proof}

For an admissible $\alpha$, let us define the \emph{$\alpha$-set register machine} by relativizing its definition over $L_\alpha$. 
By relativizing the above results to an admissible $\alpha$, we get the following:
\begin{corollary}\label{Lemma: Sigma-alpha-1 equals to alpha-TM-halts}
\pushQED{\qed}
Let $p$ be an $\alpha$-SRM. Then the statement `The program $p$ with input $x$ halts' is $\Sigma^\alpha_1$. Conversely, every $\Sigma_1$ formula $\varphi(x)$ is equivalent to a statement of the form `An $\alpha$-SRM $p$ with an input $x$ halts.'

Furthermore, the statement `The program $p$ with an input $x$ halts' is equivalent to a statement of the form $\exists d \; \mathtt{T}^\alpha(p,d,x)$ for some $\Delta_1^\alpha$ formula $\mathtt{T}^\alpha$, and satisfies if $d\in L_\beta$ for $\beta<\alpha$, then $x$ and the output of $p$ with input $x$ are in $L_\beta$.
\qedhere 
\end{corollary}

\begin{corollary}
    A partial function $F\colon L_\alpha\to L_\alpha$ is $\Sigma_1^\alpha$-definable if and only if there is an $\alpha$-SRM $p$ such that $F(x)=p(x)$ for all $x\in L_\alpha$.
    
    In fact, we can effectively determine $p$ from the $\Sigma_1^\alpha$-definition of $F$, and if $F$ is $\Sigma_1^\alpha$-definable without parameters, then the corresponding $p$ also does not have other parameters.
\end{corollary}

That is, every $\Sigma^\alpha_1$-set is  $\alpha$-register machine computably enumerable, and every $\Sigma^\alpha_1$-definable function is computed by an $\alpha$-SRM program.

Now we fix a $\KP_1$-provably admissible representable $\alpha$ with a representation $\bfalpha$. 
\begin{definition} Let $T$ be a $\Sigma^\bfalpha_1$-sound theory extending $\KP_1$ + $\Exists(\bfalpha)$.
    \begin{enumerate}
        \item Let $p$ be an $\bfalpha$-SRM. Define $\height(p)$ as the least ordinal $\beta<\alpha$ such that for every $x\in L_\beta$, if a computation $d$ of $p$ with an input $x$ halts, then $d\in L_\beta$.
        
        \item $|T|_{\bfalpha\text{-ht}}$ is the supremum of all $\height(p)$ for an $\alpha$-SRM computing a constant function such that $T$ proves $p$ halts.
        
        \item $|T|_{\bfalpha\text{-cl}}$ is the least $\beta$ satisfying $f[L_\beta]\subseteq L_\beta$ for all $T$-provably $\bfalpha$-recursive $f$ definable without parameters.
        
        Equivalently, $|T|_{\bfalpha\text{-cl}}$ is the least $\beta$ such that for every $\alpha$-SRM $p$, if $T$ proves $p$ halts for every input $x\in L_\alpha$, then the output of $p$ with input in $L_\beta$ is also in $L_\beta$.  
    \end{enumerate} 
\end{definition}

Following the lead of \cite[Lemma 1.2.3]{Pohlers1998Subsystems}, we get the following:
\begin{proposition}
     Let $\bfalpha$ be a representable ordinal that is provably admissible over $\KP_1$ and let $T$ be a $\Sigma^\bfalpha_1$-sound theory extending $\KP_1 + \Exists(\bfalpha)$.
     Then $|T|_{\Sigma^\bfalpha_1}=|T|_{\bfalpha\text{-}{\rm ht}}=|T|_{\bfalpha\text{-{\rm cl}}}=|T|_{\Pi^\bfalpha_2}$.
\end{proposition}
\begin{proof}
    \cite{Pohlers1998Subsystems}, Lemma 1.2.3 proves $|T|_{\Sigma^\bfalpha_1}=|T|_{\Pi^\bfalpha_2}$, so let us prove the following:

        $|T|_{\Sigma^\bfalpha_1}\le |T|_{\bfalpha\text{-ht}}$: To show this, it suffices to show $|T|_{\bfalpha\barrepr} \le |T|_{\bfalpha-ht}$. 
        Let $\bfgamma$ be a representation such that $T\vdash \exists \Exists(\bfgamma)^{L_\bfalpha}$. By \autoref{Lemma: Sigma-alpha-1 equals to alpha-TM-halts}, we can construct a following $\bfalpha$-SRM $p$:
        $p$ tries to find an ordinal $\xi<\alpha$ such that $R^\bfgamma(\xi)$ holds, and halts if there is such ordinal and returns that ordinal. We can decide $R^\bfgamma(\xi)$ holds within time $<\bfalpha$ by computing $R^\bfgamma(\xi)$ and $\lnot R^\bfgamma_*(\xi)$ simultaneously, where $R^\bfgamma_*$ is the $\Pi_1$ statement for $R^\bfgamma$ defined in \eqref{Formula: Pi 1 representation}.
        One of them must halt within time $<\bfalpha$, and we can decide the validity of $R^\bfgamma(\xi)$ by using that computational result.
        
        Since $T$ proves $\Exists(\bfgamma)^{L_\bfalpha}$, $T$ proves $p$ halts, and $\height(p)=\xi+1$ for the unique $\xi$ such that $L\models R^\bfgamma(\xi)$. That is, $\xi<|T|_{\bfalpha\text{-ht}}$.
        Then the desired inequality follows from the definition of $|T|_{\bfalpha\barrepr}$.

        $|T|_{\bfalpha\text{-ht}}\le |T|_{\bfalpha\text{-cl}}$: `$p$ halts' for an $\bfalpha$-SRM $p$ computing a constant function is a special case of `$p$ halts for every input $x$.' Thus, every $T$-provably-halting $\alpha$-SRM that is computing a constant function outputs a value below $|T|_{\bfalpha\text{-cl}}$.

        $|T|_{\bfalpha\text{-cl}}\le |T|_{\Pi^\bfalpha_2}$: 
        By \autoref{Lemma: Sigma-alpha-1 equals to alpha-TM-halts}, `$p$ halts with input $x$' is equivalent to $\exists d\in L_\bfalpha \; \mathtt{T}(p,d,x)$. Thus the assertion `For every $x\in L_\bfalpha$, $p$ halts with input $x$' is $\Pi^\bfalpha_2$. 
        Thus if $T$ proves $p$ halts, then by the definition of $|T|_{\Pi^\bfalpha_2}$, then $L_{|T|_{\Pi^\bfalpha_2}}\models \forall x \; \exists d \; \mathtt{T}(p,d,x)$. By \autoref{Lemma: Sigma-alpha-1 equals to alpha-TM-halts} again, we infer that the output of $p$ belongs to $L_{|T|_{\Pi^\bfalpha_2}}$. That is, $L_{|T|_{\Pi^\bfalpha_2}}$ is closed under computation by $p$.
   
    Hence we get 
    \begin{equation*}
        |T|_{\bfalpha-ht}\le |T|_{\bfalpha-cl}\le |T|_{\Pi^\bfalpha_2}=|T|_{\Sigma^\bfalpha_1}\le|T|_{\bfalpha-ht}. \qedhere 
    \end{equation*}
\end{proof}

\section{Comparing the strength of theories}\label{comparison-theorem}
Now let us examine the set-theoretic analogue of the second main theorem in \cite{walsh2023characterizations}.
\begin{definition}
    We define \emph{provability in the presence of an oracle for $\Pi^\bfalpha_1$-truths} $\vdash^{\Pi^\bfalpha_1}$ as follows: $T\vdash^{\Pi^\bfalpha_1} \varphi$ if there is a true $\Pi^\bfalpha_1$-formula $\psi$ such that $T+\psi\vdash\varphi$. Then we define
    \begin{equation*}
        T\subseteq^{\Pi^\bfalpha_1}_{\Sigma^\bfalpha_1} U\iff \text{For all $\varphi\in\Sigma^\bfalpha_1$, if $T\vdash^{\Pi^\bfalpha_1}\varphi$, then $U\vdash^{\Pi^\bfalpha_1}\varphi$}.
    \end{equation*}

    Let's introduce some notation. The reflection principle $\RFN_{\Sigma^\bfalpha_1}(T)$ is the formula:
    \begin{equation*}
        \forall \varphi\in \Delta_0 \Big(\Pr_T\big(\exists x \in L_\bfalpha \; \varphi(x)\big)\to L_\bfalpha\models_{\Sigma_1} \exists x \varphi(x)\Big).
    \end{equation*}
    We define $T\le^{\Pi^\bfalpha_1}_{\RFN_{\Sigma^\bfalpha_1}}U$ if and only if
    \begin{equation*}
        \KP_1 + \Exists(\bfalpha) \vdash^{\Pi^\bfalpha_1} \RFN_{\Sigma^\bfalpha_1}(U)\to \RFN_{\Sigma^\bfalpha_1}(T).
    \end{equation*}
\end{definition}
Here is how we can describe $T\vdash^{\Pi^\bfalpha_1}\varphi$, for $\varphi\in\Sigma^\bfalpha_1$, from an $\bfalpha$-recursion theoretic perspective. For $\varphi\in\Sigma^\bfalpha_1$, $T\vdash^{\Pi^\bfalpha_1}\varphi$ just in case there exists an $\alpha$-program $M$ such that $T\vdash^{\Pi^\bfalpha_1}$ ``$M$ halts.'' Accordingly, we may view $\RFN_{\Sigma^\bfalpha_1}(T)$ as the following claim: Every $T$-provably halting $\bfalpha$-recursive programs actually halts over $L$.

Let us start by stating a small lemma that is extracted from the proof of \autoref{Theorem: FirstEquivalence}. This small lemma will play a pivotal role in the proofs of the main results of this section.
\begin{lemma}[$\KP_1 +\Exists(\bfalpha)$] \pushQED{\qed} \label{Lemma: SmallerOrdinal-Tprovablyexists}
    Let $\gamma$ be an ordinal with representation $\bfgamma$ such that $\gamma<|T|_{\Sigma^\bfalpha_1}$. Then $T\vdash^{\Pi^\bfalpha_1}\Exists(\bfgamma)^{L_\bfalpha}$.\qedhere
\end{lemma}

It is known that adding a true $\Sigma^1_1$-sentence to a theory does not alter its $\Pi^1_1$-reflection. That is, if $\theta$ is a true $\Sigma^1_1$ sentence and $T$ is a $\Pi^1_1$-sound theory, then $\RFN_{\Pi^1_1}(T)$ implies $\RFN_{\Pi^1_1}(T+\sigma)$.
The following result is a set-theoretic analogue of what we previously stated:
\begin{lemma} \label{Lemma: True Pi 1 alpha does not affect reflection}
    Let $\theta$ be a $\Pi_1$-sentence such that $\theta^{L_\alpha}$ is true, and $T$ be a $\Sigma^\bfalpha_1$-sound extension of $\KP_1 + \Exists(\bfalpha)$.
    Then we have the following:
    \begin{equation*}
        \KP_1 + \Exists(\bfalpha) \vdash \RFN_{\Sigma^\bfalpha_1}(T) \to \RFN_{\Sigma^\bfalpha_1}(T+\theta).
    \end{equation*}
\end{lemma}
\begin{proof}
    Let us reason over $\KP_1 + \Exists(\bfalpha)$ and suppose that $\RFN_{\Sigma^\bfalpha_1}(T)$ holds. Now suppose that for a $\Sigma_1$-sentence $\sigma$, we have
    \begin{equation*}
        T + \theta^{L_\bfalpha} \vdash \sigma^{L_\bfalpha}. 
    \end{equation*}
    Then $T\vdash (\theta\to\sigma)^{L_\bfalpha}$, and $\theta\to\sigma$ is $\Sigma_1$. Hence by $\RFN_{\Sigma^\bfalpha_1}(T)$, we get $L_\alpha\vDash \theta\to \sigma$. Since $\theta^{L_\alpha}$ holds, we have $L_\alpha\vDash \sigma$.
\end{proof}

\begin{theorem}
    Let $T$ and $U$ be $\Pi^\bfalpha_1$-definable $\Sigma^\bfalpha_1$-sound extensions of $\KP_1 +\mathsf{Exists}(\boldsymbol{\alpha})$, where $\bfalpha$ is a $\KP_1$-provably admissible representation. Then we have
    \begin{equation*}
        T\subseteq^{\Pi^\bfalpha_1}_{\Sigma^\bfalpha_1} U \iff |T|_{\Sigma^\bfalpha_1}\le |U|_{\Sigma^\bfalpha_1}.
    \end{equation*}
\end{theorem}
\begin{proof}
    For the left-to-right, suppose that $T\subseteq^{\Pi^\bfalpha_1}_{\Sigma^\bfalpha_1} U$. 
    Let $\gamma$ be a representable ordinal with a representation $\bfgamma$ such that $\gamma<|T|_{\Sigma^\bfalpha_1}$.
    Then by \autoref{Lemma: SmallerOrdinal-Tprovablyexists}, we have $T\vdash^{\Pi^\bfalpha_1}\Exists(\bfgamma)^{L_\bfalpha}$.
    Since $T\subseteq^{\Pi^\bfalpha_1}_{\Sigma^\bfalpha_1} U$ and $\Exists(\bfgamma)^{L_\bfalpha}$ is $\Sigma^\bfalpha_1$, we have $U\vdash^{\Pi^\bfalpha_1}\Exists(\bfgamma)^{L_\bfalpha}$.

    Hence we have a $\Pi_1$-sentence $\psi$ such that $L_\alpha\vDash \psi$ and $U\vdash \psi^{L_\bfalpha}\to \Exists(\bfgamma)^{L_\bfalpha}$.
    Since the statement $\psi\to \Exists(\bfgamma)$ is $\Sigma_1$, we have $L_{|U|_{\Sigma^\bfalpha_1}}\models \psi\to \Exists(\bfgamma)$.
    However, since $\psi$ is $\Pi_1$, $L_\alpha\models \psi$ implies $L_{|U|_{\Sigma^\bfalpha_1}}\models \Exists(\bfgamma)$. This shows $\gamma<|U|_{\Sigma^\bfalpha_1}$.

    For the right-to-left direction, assume that $|T|_{\Sigma^\bfalpha_1}\le |U|_{\Sigma^\bfalpha_1}$ and assume that $T\vdash^{\Pi^\bfalpha_1}\varphi^{L_\bfalpha}$ for a $\Sigma_1$-formula $\varphi$. 
    Thus we can find a $\Pi_1$-formula $\psi$ such that $L_\alpha\models \psi$ and $T\vdash(\psi\to\varphi)^{L_\bfalpha}$.
    We can see that $\psi\to \varphi$ is $\Sigma_1$, so by \autoref{Lemma: alpha-Kleene}, we can find a representation $\bfgamma$ such that
    \begin{equation*}
        \KP_1 \vdash \forall\beta\Big(\text{$\beta$ is admissible}\to [(\psi\to\varphi)^{L_\beta} \leftrightarrow \Exists(\bfgamma)^{L_\beta}]\Big).
    \end{equation*}
    Since $\bfalpha$ is $\KP_1$-provably admissible, we have
    \begin{equation*}
        \KP_1 + \Exists(\bfalpha) \vdash (\psi\to\varphi)^{L_\bfalpha} \leftrightarrow \Exists(\bfgamma)^{L_\bfalpha}.
    \end{equation*}
    Since $T$ is an extension of $\KP_1 + \Exists(\bfalpha)$, we have $T\vdash \Exists(\bfgamma)^{L_\bfalpha}$.
    From this, we get $\bfgamma<|T|_{\Sigma^\bfalpha_1}\le |U|_{\Sigma^\bfalpha_1}$ by \autoref{Lemma: equality-invariants}.
    Thus by \autoref{Lemma: SmallerOrdinal-Tprovablyexists}, $U\vdash^{\Pi^\bfalpha_1} \Exists(\bfgamma)^{L_\bfalpha}$, which implies $U\vdash^{\Pi^\bfalpha_1}(\psi\to\phi)^{L_\bfalpha}$.
    Since $L_\bfalpha\models \psi$, we finally have $U\vdash^{\Pi^\bfalpha_1}\phi^{L_\bfalpha}$.
\end{proof}

Now let us turn our view to the remaining equivalence. Before continuing, let us formulate additional reflection principles:
\begin{definition}
    $\RFN_{\bfalpha\barrepr}(T)$ is the following principle: For every representation $\bfgamma$, if $T\vdash\Exists(\bfgamma)^{L_\bfalpha}$, then $L_\bfalpha\models \Exists(\bfgamma)$. 
\end{definition}
\begin{definition}
    The principles $\RFN^{\Pi^\bfalpha_1}_{\Sigma^\bfalpha_1}(T)$ and $\RFN^{\Pi^\bfalpha_1}_{\bfalpha\barrepr}(T)$ are obtained from replacing the usual $T$-provability ($T\vdash$) with $\Pi^\bfalpha_1$-provability ($T\vdash^{\Pi^\bfalpha_1}$) in the corresponding reflection principles.
\end{definition}

\begin{definition}
    Let $T$ and $U$ be $\Sigma^\bfalpha_1$-sound extensions of $\KP_1+\Exists(\bfalpha)$. We say that $T\le_{\RFN_{\Sigma^\bfalpha_1}}^{\Pi^\bfalpha_1} U$ if the following holds:
    \begin{equation*}
        \KP_1 + \Exists(\bfalpha) \vdash^{\Pi^\bfalpha_1} \RFN_{\Sigma^\bfalpha_1}(U)\to\RFN_{\Sigma^\bfalpha_1}(T).
    \end{equation*}
\end{definition}

\begin{lemma}\label{Lemma: Equivalence-RFNs}
    Suppose that $T$ is a $\Sigma^\bfalpha_1$-sound $\Delta^\bfalpha_1$-definable extension of $\KP_1+\Exists(\bfalpha)$ for a $\KP_1$-provably admissible $\bfalpha$.
    Then $\KP_1 + \Exists(\bfalpha)$ proves all of the following are all equivalent:
    \begin{enumerate}
        \item $\RFN_{\Sigma^\bfalpha_1}(T)$,
        \item $\RFN^{\Pi^\bfalpha_1}_{\Sigma^\bfalpha_1}(T)$,
        \item $\RFN_{\bfalpha\barrepr}(T)$,
        \item $\RFN^{\Pi^\bfalpha_1}_{\bfalpha\barrepr}(T)$,
    \end{enumerate}
\end{lemma}
\begin{proof}
    In this proof, we reason in $\KP_1 + \Exists(\bfalpha)$ \emph{with no other true $\Pi^\bfalpha_1$ sentences} unless specified.
    
    (1) $\leftrightarrow$ (2). $\RFN^{\Pi^\bfalpha_1}_{\Sigma^\bfalpha_1}(T)\to \RFN_{\Sigma^\bfalpha_1}(T)$ follows from the trivial fact that $T\vdash\phi$ implies $T\vdash^{\Pi^\bfalpha_1}\phi$.
    
    For the converse, suppose that $T\vdash^{\Pi^\bfalpha_1}\varphi^{L_\bfalpha}$ for some $\Sigma_1$-statement $\varphi$. Then we have a $\Pi_1$-sentence $\psi$ such that $L_\alpha\models \psi$ and $T\vdash\psi^{L_\bfalpha}\to\varphi^{L_\bfalpha}$.
    Hence by $\RFN_{\Sigma^\bfalpha_1}(T)$ and from the fact that $\psi^{L_\alpha}\to\varphi^{L_\alpha}$ is $\Sigma^\bfalpha_1$, we have $L_\alpha\models \psi\to\varphi$. But we know that $L_\alpha\models\psi$, so $L_\alpha\models\varphi$.
    That is, $T\vdash^{\Pi^\bfalpha_1}\varphi^{L_\bfalpha}$ implies $L_\alpha\models \phi$, which is the very statement of $\RFN^{\Pi^\bfalpha_1}_{\Sigma^\bfalpha_1}(T)$.

    (3) $\leftrightarrow$ (4). $\RFN^{\Pi^\bfalpha_1}_{\bfalpha\barrepr}(T)\to \RFN_{\bfalpha\barrepr}(T)$ is clear as before.
    To show the converse, suppose that $T\vdash^{\Pi^\bfalpha_1}(\Exists(\bfgamma))^{L_\bfalpha}$ for a representation $\bfgamma$.
    Then we can find a $\Pi_1$-sentence $\psi$ such that $L_\alpha\models\psi$ and
    \begin{equation} \tag{$B$}\label{Formula: T proves psi to Exists gamma}
        T\vdash \psi^{L_\bfalpha}\to \Exists(\bfgamma)^{L_\bfalpha}.
    \end{equation}
    By \autoref{Lemma: ordinal-Kleene}, we can find a representation $\bfdelta$ such that
    \begin{equation}\tag{$C$}\label{Formula: Lemma-Equivalence-RFNs}
        \KP_1 +(V=L) \vdash  \Exists(\bfdelta)\leftrightarrow [\psi\to \Exists(\bfgamma)].
    \end{equation}
    We know that $\bfalpha$ is a $\KP_1$-provably admissible representation. Which is to say that we know that:
    \begin{equation*}
        \KP_1 + \Exists(\bfalpha) \vdash[ L_\bfalpha\vDash \KP_1 + (V=L)].
    \end{equation*}
    Hence by \eqref{Formula: Lemma-Equivalence-RFNs}, we have   
    \begin{equation}\tag{$D$}\label{Formula: Lemma-Equivalence-RFNs 2}
        \KP_1 + \Exists(\bfalpha) \vdash \Exists(\bfdelta)^{L_\bfalpha}\leftrightarrow [\psi^{L_\bfalpha}\to \Exists(\bfgamma)^{L_\bfalpha}].
    \end{equation}
    Since $T$ extends $\KP_1 + \Exists(\bfalpha)$, we have $T\vdash \Exists(\bfdelta)^{L_\alpha}$ from \eqref{Formula: T proves psi to Exists gamma} and \eqref{Formula: Lemma-Equivalence-RFNs 2}.
    From $\RFN_{\bfalpha\barrepr}$, we have $L_\alpha\models \Exists(\bfdelta)$, which is equivalent to $L_\alpha\models \psi\to \Exists(\bfgamma)$ by \eqref{Formula: Lemma-Equivalence-RFNs}.
    
    Since $L_\alpha\models \psi$, we have $L_\alpha\models \Exists(\bfgamma)$. In sum, $T\vdash^{\Pi^\bfalpha_1}(\Exists(\bfgamma))^{L_\bfalpha}$ implies $L_\alpha\models \Exists(\bfgamma)$.
    
    (1) $\leftrightarrow$ (3). $\RFN_{\Sigma^\bfalpha_1}(T)\to \RFN_{\bfalpha\barrepr}(T)$ follows from the fact that $\Exists(\bfgamma)$ is a $\Sigma_1$ sentence.
    
    Now let us show the converse.
    For a $\Sigma_1$-sentence $\varphi$, we can find a representation $\bfgamma$ such that $\varphi^{L_\bfalpha}\leftrightarrow \Exists(\bfgamma)^{L_\bfalpha}$, which follows from \autoref{Lemma: alpha-Kleene} and by repeating the previous arguments.
    
    Now suppose that $T\vdash\varphi^{L_\bfalpha}$, then $T\vdash \Exists(\bfgamma)^{L_\bfalpha}$. We assumed $\RFN_{\bfalpha\barrepr}(T)$, so we infer that $L_\bfalpha\models \Exists(\bfgamma)$. Since $\phi^{L_\bfalpha}\leftrightarrow \Exists(\bfgamma)^{L_\bfalpha}$, we get $L_\bfalpha\models\phi$.
\end{proof}

The next theorem shows that comparing the $\Sigma^\bfalpha_1$-reflection order modulo $\Pi^\bfalpha_1$-provability coincides with the ordering induced by comparing $\Sigma^\bfalpha_1$-proof-theoretic ordinals.
\begin{theorem}\label{well-ordering-thm}
    Let $T$ and $U$ be $\Delta^\bfalpha_1$-definable $\Sigma^\bfalpha_1$-sound extensions of $\KP_1 + \mathsf{Exists}(\bfalpha)$. Then we have
    \begin{equation}
        |T|_{\Sigma^\bfalpha_1}\le |U|_{\Sigma^\bfalpha_1} \iff T\le^{\Pi^\bfalpha_1}_{\RFN_{\Sigma^\bfalpha_1}}U.
    \end{equation}
\end{theorem}
\begin{proof}
    Let us prove the left-to-right implication first. It suffices to show that the following holds:
    \begin{equation*}
        \KP_1 + \Exists(\bfalpha) \vdash^{\Pi^\bfalpha_1} \RFN^{\Pi^\bfalpha_1}_{\bfalpha\barrepr}(U)\to \RFN_{\bfalpha\barrepr}(T).
    \end{equation*}
    Then the desired implication follows from \autoref{Lemma: Equivalence-RFNs}.

    First, let us reason over the metatheory. Suppose that $\bfgamma$ is a representation such that $T\vdash\Exists(\bfgamma)^{L_\bfalpha}$.
    If $\bfgamma$ is a representation of $\gamma$, then $\gamma<|T|_{\Sigma^\bfalpha_1}\le|U|_{\Sigma^\bfalpha_1}$.
    Hence by \autoref{Lemma: SmallerOrdinal-Tprovablyexists}, we have $U\vdash^{\Pi^\bfalpha_1}\Exists(\bfgamma)^{L_\bfalpha}$.
    In short, we have shown that the following sentence is \emph{true}:
    \begin{equation*}
        \theta :\equiv \forall \bfgamma \Big(\text{$\bfgamma$ is a representation}\to [\Pr_T(\Exists(\bfgamma)^{L_\bfalpha})\to \Pr_U^{\Pi^\bfalpha_1}(\Exists(\bfgamma)^{L_\bfalpha})]\Big).
    \end{equation*}

    Since both $T$ and $U$ are $\Delta^\bfalpha_1$-definable, $\theta$ is a true $\Pi^\bfalpha_1$ sentence.\footnote{In fact, $\theta$ is $\Delta^\bfalpha_1$, but we will not use this fact in our argument.}

    Now let us reason over $\KP_1 + \Exists(\bfalpha) + \theta$. If $\bfgamma$ is a representation such that $T\vdash \Exists(\bfgamma)^{L_\bfalpha}$, then we have $U\vdash^{\Pi^\bfalpha_1} \Exists(\bfgamma)^{L_\bfalpha}$.
    Thus if we have $\RFN^{\Pi^\bfalpha_1}_{\bfalpha\barrepr}(U)$, then we get $L_\alpha\models \Exists(\bfgamma)$.
    Hence we get $\RFN_{\bfalpha\barrepr}(T)$.

    To show the right-to-left implication, assume for the sake of contradiction that the implication fails. Thus we have $T\le^{\Pi^\bfalpha_1}_{\RFN_{\Sigma^\bfalpha_1}}U$ but $|U|_{\Sigma^\bfalpha_1}<|T|_{\Sigma^\bfalpha_1}$.
    Let us fix a representation $\bfgamma$ with a value $\gamma$ such that $T\vdash \Exists(\bfgamma)^{L_\bfalpha}$ and $|U|_{\Sigma^\bfalpha_1}<\gamma$. Now consider the statement $F$ roughly saying `Every representation such that $U$ proves its value exists has a value less than $\bfgamma$':
    \begin{equation*}
        F:= \text{For every representation $\bfdelta$, } [[U\vdash \Exists(\bfdelta)^{L_\bfalpha}]\to L_\bfgamma\models \Exists(\bfdelta)].
    \end{equation*}
    Note that its more formal statement is as follows:
    \begin{equation*}
        \forall \zeta<\bfalpha \forall \bfdelta [\text{$\bfdelta$ is a representation}\land R^\bfgamma(\zeta)\to [[U\vdash \Exists(\bfdelta)^{L_\bfalpha}]\to L_\zeta\models \Exists(\bfdelta)]].
    \end{equation*}
    Then $F$ is a true $\Pi^\bfalpha_1$-statement; indeed, the formula consists of universal quantifiers ranging over a Boolean combination of formulas each of which is at most $\Delta_1^{L_\bfalpha}$ (recall that the $U$-provability relation is $\Delta_1^{L_\bfalpha}$ by assumption).

Since $T$ proves $\bfgamma < \bfalpha$, we have $T+F\vdash\RFN_{\bfalpha\barrepr}(U)$.
By \autoref{Lemma: True Pi 1 alpha does not affect reflection}, we get
\begin{equation*}
    \KP_1 +  \Exists(\bfalpha) + F \vdash \RFN_{\Sigma^\bfalpha_1}(T) \to \RFN_{\Sigma^\bfalpha_1}(T+F).
 \end{equation*}

    Combining $T+F\vdash\RFN_{\bfalpha\barrepr}(U)$ with \autoref{Lemma: Equivalence-RFNs}, we have
    \begin{equation*}
        \KP_1 +  \Exists(\bfalpha) + F \vdash \RFN_{\Sigma^\bfalpha_1}(T) \to \RFN_{\Sigma^\bfalpha_1}(T+\RFN_{\Sigma^\bfalpha_1}(U)).
    \end{equation*}
    Since $T\le^{\Pi^\bfalpha_1}_{\RFN_{\Sigma^\bfalpha_1}}U$, we can find a true $\Pi^\bfalpha_1$-sentence $G$ such that
    \begin{equation*}
        \KP_1 +  \Exists(\bfalpha) + G \vdash \RFN_{\Sigma^\bfalpha_1}(U)\to \RFN_{\Sigma^\bfalpha_1}(T).
    \end{equation*}
    By combining these two, we get
    \begin{equation*}
        \KP_1 +  \Exists(\bfalpha) + F + G \vdash \RFN_{\Sigma^\bfalpha_1}(U)\to \RFN_{\Sigma^\bfalpha_1}(T+\RFN_{\Sigma^\bfalpha_1}(U)).
    \end{equation*}
    Since $T$ contains $\KP_1+\Exists(\bfalpha)$,  we get
    \begin{equation*}
        \KP_1 +  \Exists(\bfalpha) + F + G + \RFN_{\Sigma^\bfalpha_1}(U) \vdash \RFN_{\Sigma^\bfalpha_1}(\KP_1+ \Exists(\bfalpha)+\RFN_{\Sigma^\bfalpha_1}(U)).
    \end{equation*}
    Let $S$ be the theory on the left-hand side of the above formula. Since $F$ and $G$ are true $\Pi^\bfalpha_1$-formulas, we have $S\vdash \RFN_{\Sigma^\bfalpha_1}(S)$. 
    However, $S$ is axiomatized by sentences valid over $L$, which means $S$ is sound. A contradiction.
\end{proof}

Unlike the proof in \cite{walsh2023characterizations}, we do not put extra care on proving the left-to-right because $\KP_1$ already proves we can swap an order of a bounded quantifier and an unbounded $\exists$ as long as the resulting formula is still $\Sigma_1$ by \autoref{Lemma: Formula-normalization}.

\section{Well-foundedness of $\RFN_{\Sigma^{\alpha}_1}$-comparison}\label{well-foundedness-thm}
Theorem \ref{well-ordering-thm} demonstrates that $\le^{\Pi^\bfalpha_1}_{\RFN_{\Sigma^\bfalpha_1}}$ is well-ordered. However, the relation $\le^{\Pi^\bfalpha_1}_{\RFN_{\Sigma^\bfalpha_1}}$ relies on the notion of provability \emph{in the presence of a $\Pi^\bfalpha_1$ oracle}, which is far from the actual provability relation. Can we obtain a similar result for the usual provability relation? It was proven in \cite{Walsh2022Incompleteness} that there is no sequence $\langle T_n \mid n<\omega\rangle$ of $\Pi^1_1$-sound, $\Sigma^1_1$-definable extensions of $\mathsf{\Sigma^1_1\text{-}AC_0}$ such that $T_n\vdash \RFN_{\Pi^1_1}(T_{n+1})$, i.e., the relation
\begin{equation*}
    S <_{\RFN_{\Pi^1_1}} T :\equiv T \vdash \RFN_{\Pi^1_1}(S)
\end{equation*}
is well-founded on $\Pi^1_1$-sound, $\Sigma^1_1$-definable extensions of $\mathsf{\Sigma^1_1\text{-}AC_0}$.\footnote{This is an analogue of an earlier result in \cite{pakhomov2021reflection}, which concerns  $\Pi^1_1$-sound, $\Sigma^0_1$-definable extensions of $\mathsf{ACA}_0$.} In this section, we prove its set-theoretic analogue.
\begin{theorem} \label{Theorem: RFN-comparison-WF-main}
    There is no sequence $\langle T_n\mid n<\omega\rangle$ of $\Sigma^\bfalpha_1$-sound, $\Pi^\bfalpha_1$-definable extensions of $\KP_1 + \Exists(\bfalpha)$ such that for each $n$, $T_n\vdash \RFN_{\Sigma^\bfalpha_1}(T_{n+1})$.
\end{theorem}

Our proof will and must be different from what is presented in \cite{Walsh2022Incompleteness}. The main reason is that we are using an $\alpha$-recursive characteristic function as a set-theoretic analogue of definable well-orders. However, this $\alpha$-recursive characteristic function is unary, unlike well-orders.

The following lemma is a form of diagonal lemma we will use:
\begin{lemma} \label{Lemma: Diagonal Lemma for two variables}
    Let $\phi(n,x)$ be a $\Sigma_1$-formula. Then we can find a $\Sigma_1$-formula $\psi$ such that 
    \begin{equation*}
        \KP_1 \vdash \forall x [\psi(x) \leftrightarrow \phi(\ulcorner \psi\urcorner, x)].
    \end{equation*}
\end{lemma}
\begin{proof}
    Consider the following computable function for formulas $\theta$ with two free variables:
    \begin{equation*}
        d(\ulcorner\theta\urcorner) = \ulcorner\theta(\underline{\ulcorner\theta\urcorner},x)\urcorner
    \end{equation*}
    Here $\underline{n}$ means the numeral for $n$, i.e., the symbol for $1+\cdots + 1$ with $n$ many 1. $d$ is computable, so it is $\Delta_0$-definable with a parameter $\omega$ over $\KP_1$.
    Now for a given $\Sigma_1$-formula $\phi(n,x)$, let $\phi'(n,x)$ be a $\Sigma_1$-formula that is $\KP_1$-provably equivalent to 
    \begin{equation*}
        \exists m<\omega (d(n)=m \land \phi(m,x)).
    \end{equation*}
    Then for every $\Sigma_1$-formula $\theta$ with two variables we have
    \begin{equation*}
        \KP_1 \vdash \phi'(\underline{\ulcorner\theta\urcorner},x) \leftrightarrow \phi(\ulcorner\theta(\underline{\ulcorner\theta\urcorner},x)\urcorner,x).
    \end{equation*}
    Now take $\theta\equiv \phi'$, and let $\psi(x)\equiv \phi'(\underline{\ulcorner\phi'\urcorner},x)$.
    Then we get
    \begin{equation*}
        \KP_1 \vdash \psi(x) \leftrightarrow \phi(\ulcorner\psi\urcorner,x).
    \end{equation*}
    Since $x$ occurs free, we have the desired result by universal generalization.
\end{proof}

First, we will prove a set-theoretic analogue of the $\Sigma^1_1$-boundedness lemma:
\begin{lemma}[$\KP_1 + (V=L)$] \label{Lemma: Pi1 boundness for Ord}
    Suppose that $c$ is a set and let $f\colon c\rightharpoonup  \Ord$ be a $\Sigma_1$-definable partial function.
    Furthermore, assume that for a $\Pi_1$-formula $\phi(x)$, $\operatorname{dom}(f)  \supseteq \{x\in c\mid\phi(x)\}$.
    Then the class $\{f(x)\mid x\in c\land \phi(x)\}\subseteq\Ord$ is bounded.
\end{lemma}
\begin{proof}
$\Sigma_1$ truth is $\Sigma_1$-definable over $\KP_1$ as in \autoref{Definition: Sigma n truth}. That means that for \emph{every} $\Sigma_1$ sentence $\sigma$, 
    \begin{equation}\label{truth-dfn}
        \KP_1\vdash\sigma \leftrightarrow \mathsf{True}_{\Sigma_1}(\ulcorner\sigma\urcorner, 0).
    \end{equation}
    
    Now let us \textbf{reason within} $\KP_1 + (V=L)$: Suppose, toward a contradiction, that the class $\{f(x)\mid x\in c\land \phi(x)\}$ is cofinal in the class of all ordinals.
    Now we claim that the partial truth predicate for $\Sigma_1$-formulas is $\Pi_1$-definable. In fact, we can see that the following equivalence holds for all $p$ and for each $\Sigma_1$ formula $\psi$:
    \begin{equation*}
        \vDash_{\Sigma_1}\psi(p) \iff \exists x\in c [\phi(x) \land \forall \xi (f(x)=\xi\to L_{\xi}\models \psi(p))].
    \end{equation*}
    The equivalence holds due to the cofinality of the class $\{f(x)\mid x\in c\land \phi(x)\}$, and the right-hand formula is $\Pi_1$ since the bounded quantifier does not alter the complexity of the formula (due to $\Sigma_1$-Collection), $\phi(x)$ is $\Pi_1$, and the formula
    \begin{equation*}
        \forall \xi (f(x)=\xi\to L_{\xi}\models \phi(p))
    \end{equation*}
    is $\Pi_1$, which follows from the $\Sigma_1$-definability of $f$ and the $\Delta_1$-definability of $L_\xi$.

So we have concluded that $\vDash_{\Sigma_1}$ is $\Pi_1$-definable. That is, there is a $\Pi_1$-formula $\mathsf{True}^\star_{\Sigma_1}(e,x)$ such that
    \begin{equation}\label{truth-equiv}
       \forall \varphi \in \Sigma_1 \forall x \Big( \mathsf{True}^\star_{\Sigma_1}(\varphi,x) \leftrightarrow  \mathsf{True}_{\Sigma_1}(\varphi,x) \Big)
    \end{equation}

By \autoref{Lemma: Diagonal Lemma for two variables}, there is a $\Sigma_1$-sentence $\lambda$ such that:
    \begin{equation}\label{fixed-point}
          \lambda \iff \lnot \mathsf{True}^\star_{\Sigma_1}(\ulcorner\lambda\urcorner, 0).
    \end{equation}
So \eqref{fixed-point} and \eqref{truth-equiv} jointly entail:
\begin{flalign*}
    \lambda \iff \neg \mathsf{True}_{\Sigma_1}(\lambda,0)
\end{flalign*}
On the other hand, since $\lambda$ is $\Sigma_1$, \eqref{truth-dfn} entails:
\begin{equation*}
     \lambda \iff \mathsf{True}_{\Sigma_1}(\lambda,0).
\end{equation*}
A contradiction. Hence $$\{f(x)\mid x\in c\land \phi(x)\}$$ must be bounded.
\end{proof}

\begin{remark}
    Much care is taken in \cite{Walsh2022Incompleteness} to stick closely to Gentzen-style methods and to avoid diagonalization. That paper appeals to the $\Sigma^1_1$-bounding. However, Beckmann and Pohlers proved $\Sigma^1_1$-bounding via analysis of cut-free derivations \cite{beckmann1998applications}, whence appeals to $\Sigma^1_1$-bounding can still be considered ``diagonalization-free.'' This raises a question: Is \autoref{Lemma: Pi1 boundness for Ord} provable via Gentzen-style methods and without the use of diagonalization?
\end{remark}

By applying the above lemma to $L_\alpha$, which is a model of $\KP_1 + (V=L)$, we get the following:
\begin{lemma}[$\KP_1 + \Exists(\bfalpha)$] \label{Lemma: Pi1 boundness for alpha} \pushQED{\qed}
    Suppose that $c\in L_\alpha$ and let $f$ be a $\Sigma^\bfalpha_1$-definable partial function from $c$ to $\alpha$.
    Furthermore, assume that for a $\Pi^\bfalpha_1$-formula $\phi(x)$, $f$ is defined over $b:=\{x\in c\mid \phi(x)\}$ where $b$ is not necessarily a member of $L_\alpha$.
    Then there is $\gamma<\alpha$ such that for every $x\in b$, $f(x)<\gamma$. \qedhere 
\end{lemma}

Now we prove the set-theoretic analogue of the key lemma (Lemma 3.4 of \cite{Walsh2022Incompleteness}):
\begin{lemma}\label{Lemma: KeyLemma-RFN-comparison-WF}
    Let $T$ be a $\Sigma^\bfalpha_1$-sound, $\Pi^\bfalpha_1$-definable extension of $\KP_1 + \Exists(\bfalpha)$.
    Then we can find a $\Sigma_1$-formula $\phi_T(x)$ such that
    \begin{equation*}
        \KP_1 + \Exists(\bfalpha)\vdash \RFN_{\Sigma^\bfalpha_1}(T)\to \exists \xi<\bfalpha \phi_T^{L_\bfalpha}(\xi),
    \end{equation*}
    and 
    $|T|_{\Sigma^\bfalpha_1}$ is the least ordinal $\xi$ satisfying $\phi_T^{L_\alpha}(\xi)$.
\end{lemma}
\begin{proof}
    Let $\theta$ be a $\Pi_1$ formula such that $\theta^{L_\bfalpha}(\ulcorner\psi\urcorner)$ is equivalent to $T\vdash \psi$.
    Let $\phi_T(\xi)$ be the assertion `Every $T$-provable $\Sigma_1^\bfalpha$-sentence is valid over $L_\xi$,' which will have a role of $\prec_T$ in the proof of \cite{Walsh2022Incompleteness}, Lemma 3.4:
    \begin{equation*}
        \phi_T(\xi):= \forall \sigma\in \Sigma_1 (\theta(\sigma^{L_\bfalpha})\to L_\xi\models \sigma).
    \end{equation*}
    Since $\theta$ is $\Pi_1$, $\phi_T^{L_\bfalpha}(\xi)$ is $\Sigma^\bfalpha_1$. By the definition of $|T|_{\Sigma^\bfalpha_1}$, the least ordinal $\xi$ satisfying $\phi_T^{L_\bfalpha}(\xi)$ is $|T|_{\Sigma^\bfalpha_1}$.

    Now working inside $\KP_1 + \Exists(\bfalpha)$, for the sake of contradiction, suppose that $\RFN_{\Sigma^\bfalpha_1}(T)$ holds but $\exists \xi<\alpha\phi_T^{L_\bfalpha}(\xi)$ fails.
    That is, for each $\xi<\alpha$, we can find a $\Sigma_1$ sentence $\sigma$ such that $T\vdash \sigma^{L_\bfalpha}$ but $L_\xi\models \lnot\sigma$.
    Now take 
    \begin{equation*}
        b=\{\sigma\in \Sigma_1\mid T\vdash \sigma^{L_\alpha}\}.
    \end{equation*}
    $b$ may not be an element of $L_\alpha$, but it is a subset of the set of $\Sigma_1$-formulas, and we can view the set of all $\Sigma_1$-formulas as a recursive subset of $\omega$.
    That is, $b$ is separated by a $\Pi^\bfalpha_1$-formula from a set in $L_\alpha$.
    Furthermore, if we take $f$ to be
    \begin{equation*}
        f(\sigma) = \text{the least $\xi<\alpha$ such that }L_\xi\models \sigma \text{ if it exists,}
    \end{equation*}
    for a $\Sigma_1$ sentence $\sigma$, then $f$ is a partial $\Sigma^\bfalpha_1$ function.

    By $\RFN_{\Sigma^\bfalpha_1}$, we can see that $f$ is defined for all $\Sigma_1$-sentence $\sigma\in b$. However, by our assumption, for every $\xi<\alpha$, we can find $\sigma\in b$ such that $f(\sigma)>\xi$. This contradicts \autoref{Lemma: Pi1 boundness for alpha}.
\end{proof}

Let us remind the reader that we have argued that $T$-provability is also $\Pi^\bfalpha_1$ because definable objects over $\omega$ are $\alpha$-recursive for $\alpha\ge \omega_1^{\mathsf{CK}}$.
If $\alpha=\omega$, then the $T$-provability relation will be $\Sigma^\omega_2$.

Then the proof of \autoref{Theorem: RFN-comparison-WF-main} follows almost immediately:
\begin{proof}[Proof of \autoref{Theorem: RFN-comparison-WF-main}]
    Suppose that $\langle T_n\mid n<\omega\rangle$ is a sequence of $\Sigma^\bfalpha_1$-sound, $\Pi^\bfalpha_1$-definable extensions of $\KP_1 + \Exists(\bfalpha)$ such that for each $n$, $T_n\vdash \RFN_{\Sigma^\bfalpha_1}(T_{n+1})$.
    By \autoref{Lemma: KeyLemma-RFN-comparison-WF} and that $T_n\supseteq \KP_1 + \Exists(\bfalpha)$, we have $\Sigma_1$ formulas $\phi_{T_n}$ such that
    \begin{equation*}
        T_n\vdash \RFN_{\Sigma^\bfalpha_1}(T_{n+1})\to \exists\xi<\bfalpha \phi^{L_\bfalpha}_{T_{n+1}}(\xi).
    \end{equation*}
    Hence $T_n\vdash \exists \xi<\bfalpha \phi^{L_\bfalpha}_{T_{n+1}}(\xi)$.
    By the definition of $|T|_{\Sigma^\bfalpha_1}$, we have $L_{|T|_{\Sigma^\bfalpha_1}}\models \exists \xi \phi_{T_{n+1}}(\xi)$.
    
    Since $\phi_{T_{n+1}}$ is $\Sigma_1$, $L_{|T_n|_{\Sigma^\bfalpha_1}}\models \exists\xi \phi_{T_{n+1}}(\xi)$ implies $\exists \xi< |T_n|_{\Sigma^\bfalpha_1}\ L_\alpha\models\phi_{T_{n+1}}(\xi)$, which is equivalent to
    \begin{equation*}
        \exists \xi<|T_n|_{\Sigma^\bfalpha_1} \phi_{T_{n+1}}^{L_\alpha}(\xi),
    \end{equation*}
    so $|T_{n+1}|_{\Sigma^\bfalpha_1}<|T_n|_{\Sigma^\bfalpha_1}$. It holds for all $n$, so we have an infinite decreasing sequence of ordinals, a contradiction.
\end{proof}

\begin{remark}
    A version of G\"{o}del's second incompleteness follows as a special case, namely, $T$ cannot prove its own reflection principle $\RFN_{\Sigma_1^\bfalpha}(T)$ assuming that $T$ is a $\Sigma^\bfalpha_1$-sound $\Pi^\bfalpha_1$-definable extension of $\KP_1 + \Exists(\bfalpha)$. To see that this is an analogue of G\"{o}del's theorem, note that another way of stating G\"{o}del's theorem is that $T$ cannot prove its own reflection principle $\RFN_{\Pi^0_1}(T)$ assuming that $T$ is a $\Pi^0_1$-sound $\Sigma^0_1$-definable extension of $\mathsf{EA}$; see the discussion in \cite[\textsection 1]{Walsh2022Incompleteness}.
\end{remark}

\section{Higher pointclasses}\label{discussion}

Throughout this paper, we gave various characterizations of generalized ordinal analysis and reflection principles in the context of set theories. There are some points that we have not addressed in detail so far, but that are still worth discussing. First, we will discuss connections between our results and inductive definitions. Then we will discuss how to generalize our results to theories inconsistent with the axiom $V=L$.

\subsection{Inductive definitions}

In \autoref{alpha-recursion} we proved that $|T|_{\Sigma^\bfalpha_1}$ gauges the complexity of the $T$-provably total $\bfalpha$-recursive function. In this subsection, we will discuss the connections between $|T|_{\Sigma^\bfalpha_1}$ and inductive definitions. The classic theory of inductive definitions yields a ``bottom up'' perspective on the $\Pi^1_1$-definable sets. Pohlers has emphasized in various places that iterating inductive definitions yields a very fine-grained stratification of sets of natural numbers into pointclasses. For instance, iterated inductive definitions yield a fine-grained stratification of the $\Delta^1_2$ sets. In this subsection we will show that our generalized notions of ordinal analysis provide an alternate perspective on these fine-grained pointclasses.

First, let's briefly introduce inductive definitions. 

Let $\mathfrak{M}$ be a structure and $\phi(n, X)$ be an $\mathfrak{M}$ formula except with a set variable $X$ occurring positively. Then we may view $\phi$ as a monotone operator 
$$\Gamma_\phi(X)=\{n\mid \mathfrak{M}\models \phi(n,X)\}.$$
For a class $\mathbf{\Gamma}$ of formulas, we say that a set $X\subseteq \textsf{domain}(\mathfrak{M})$ is \emph{$\mathbf{\Gamma}$-inductive} if $X$ is the least fixed point $I_\phi$ of an operator $\Gamma_\phi$ for $\phi\in\mathbf{\Gamma}$. In previous work, other logicians have been particularly concerned with the structure $\mathbb{N}$ of natural numbers in an expansion of the language of arithmetic.
If $\mathbf{\Gamma}$ is the set of all first-order formulas over $\mathfrak{M}$, we call the $\mathbf{\Gamma}$-Inductive sets the \emph{inductive} sets.

We have just defined inductive sets from a top-down point-of-view, but there is also a standard understanding of inductive sets from a bottom-up point-of-view: Let $\phi(n, X)$ be a first-order formula with $X$ occurring positively, and let $\Gamma_\phi$ be the corresponding monotone operator.
Consider $I^\xi_\phi$ defined recursively by $I^\xi_\phi = \bigcup_{\eta<\xi} \Gamma_\phi(I^\eta_\phi)$.
Then there is a least ordinal $\gamma$ such that $I^\gamma_\phi=I^{\gamma+1}_\phi$, and $I^\gamma_\phi$ is an inductive set $I_\phi$ generated by $\phi$. We call $\gamma$ the \emph{closure ordinal of $\Gamma_\phi$}. Also, for $n\in I_\phi$, define the \emph{$\phi$-norm} by
\begin{equation*}
    |n|_\phi = \min \{\xi \mid n\in I^{\xi+1}_\phi\}.
\end{equation*}

It is well-known (cf. \cite{Moschovakis1974ElementaryInd}) that the inductive sets over $\mathbb{N}$ are precisely the $\Pi^1_1$-sets, which are also precisely the $\Sigma_1^{\omega_1^\mathsf{CK}}$-sets.
Furthermore, $|T|_{\omega_1^\mathsf{CK}}$ is the supremum of all $\phi$-norms of $n$ that is $T$-provably in $I_\phi$ (cf. \cite[\S6]{Pohlers2009ProofTheory}):
\begin{equation*}
    |T|_{\omega_1^\mathsf{CK}} = \sup\{|n|_\phi \mid \gamma\text{ is a closure ordinal of an arithmetical $\phi$ and }T\vdash n\in I_\phi\}.
\end{equation*}

We can climb beyond the $\Pi^1_1$ sets to higher pointclasses by iterating inductive definitions. For instance, if $\mathbf{\Gamma}$ is the pointclass of all inductive sets, then the $\mathbf{\Gamma}$-Inductive sets go beyond the pointlcass of $\Pi^1_1$ sets.
Pohlers \cite{Pohlers2022Performance} formulated a framework to describe sets defined by iterated inductive definitions by iterating the notion of \emph{Spector classes} due to Moschavovakis \cite{Moschovakis1974ElementaryInd} (see \cite[\S 3.3]{Pohlers2022Performance} for Pohlers' definitions).

\begin{definition}
    A \emph{Spector class} $\mathcal{S}$ over a structure $\mathfrak{M}=(M,R\in \mathcal{R},f\in\mathcal{F})$ is the collection $\mathbf{\Gamma}=\bigcup_{n<\omega}\mathbf{\Gamma}^n$ of all relations over $M$ satisfying the following conditions:
    \begin{enumerate}
        \item $\mathbf{\Gamma}$ is closed under Boolean combinations, first-order quantification, and trivial combinatorial substitutions.\footnote{Trivial substitution is a map over $\mathfrak{M}^n\times \mathcal{P}(\mathfrak{M})^m$ to itself permuting or repeating components, that is, a map of the form 
        \begin{equation*}
            (x_0,\cdots, x_{n-1}, A_0,\cdots, A_{m-1})
            \mapsto
            (x_{i_0},\cdots, x_{i_{n-1}}, A_{j_0},\cdots, A_{j_{m-1}})
        \end{equation*}
        for $0\le i_k < n$ and $0\le j_k<m$.
        }
        \item $\mathbf{\Gamma}$ is universal: For each $n$, there is an $(n+1)$-ary $U\in \mathbf{\Gamma}^{n+1}$ such that for every $n$-ary $R\in \mathbf{\Gamma}^n$ we can find $e\in M$ such that $\mathfrak{M}\models \forall \vec{x} [R(\vec{x})\leftrightarrow U(e,\vec{x})]$.
        \item $\mathbf{\Gamma}$ is normed: For every $R\in\mathbf{\Gamma}$ there is an ordinal $\lambda$ and a norm $\sigma_R\colon R\to \lambda$. Furthermore, we have stage comparison relations $<^*_{\sigma_R},\le^*_{\sigma_R} \in \mathbf{\Gamma}$ satisfying
        \begin{equation*}
            \vec{m} <^*_{\sigma_R} \vec{n} \iff R(\vec{m})\land [R(\vec{n})\to \sigma(\vec{m})< \sigma(\vec{n})],
        \end{equation*}
        \begin{equation*}
            \vec{m} \le^*_{\sigma_R} \vec{n} \iff R(\vec{m})\land [R(\vec{n})\to \sigma(\vec{m})\le \sigma(\vec{n})].
        \end{equation*}
        \item $\mathcal{R\cup F}\subseteq \mathbf{\Gamma}\cap\check{\mathbf{\Gamma}}$, where $\check{\mathbf{\Gamma}} = \bigcup \{R\subseteq M^n\mid M^n\setminus R\in \mathbf{\Gamma}^n\}$.
    \end{enumerate}

    For each structure $\mathfrak{M}$, $\mathbb{SP}(\mathfrak{M})$ is the intersection of all Spector classes over $\mathfrak{M}$.
\end{definition}

The class of all inductive sets is an example of a Spector class, and in fact, the least Spector class over $\mathbb{N}$.
In general, for an acceptable%
\footnote{Roughly, a structure is acceptable if it admits a definable tupling function $x_0,\cdots, x_n\mapsto \langle x_0,\cdots, x_n\rangle$ and its decoding functions.} 
structure $\mathfrak{M}$ over a finitary language, $\mathbb{SP}(\mathfrak{M})$ is precisely the set of all inductive sets over $\mathfrak{M}$ (Corollary 9A.3 of \cite{Moschovakis1974ElementaryInd}).

\begin{definition}
    Let $\mathfrak{M}$ be an acceptable
    structure. Define
    \begin{itemize}
        \item $\mathbb{SP}_\mathfrak{M}^0 = \varnothing$, 
        \item $\mathbb{SP}_\mathfrak{M}^{\xi+1} = \mathbb{SP}((\mathfrak{M}; \mathbb{SP}^\xi_\mathfrak{M}))$,
        \item $\mathbb{SP}^\delta_\mathfrak{M}=\bigcup_{\xi<\delta} \mathbb{SP}^\xi_\mathfrak{M}$ if $\delta$ is a limit.
    \end{itemize}
\end{definition}

$\mathbb{SP}^\xi_\mathbb{N}$ should give a fine division of pointclasses, which corresponds to pointsets that are many-one reducible to the $\xi$th iterate of hyperjumps of 0.
$\mathbb{SP}^1_\mathbb{N}$ is precisely the collection of all inductive sets, or equivalently, lightface $\Pi^1_1$-sets.
For a recursive ordinal $\xi$, we may view $(\mathbb{N};\mathbb{SP}^\xi_\mathbb{N})$ as a finitary structure%
\footnote{Though Pohlers (cf. \cite{Pohlers2022Performance}, Theorem 3.32.) cites Moschovakis, he does not provide an explanation of how to circumvent Moschovakis' technical restriction to finitary languages. Thus, let us provide a sketch of how to do so: 
Take a $\Sigma^0_1$-definable $D\subseteq \mathbb{N}$ and $R\subseteq D^2$ such that $R$ is a well-order of order-type $\alpha$. For each $\xi<\alpha$ choose a universal set $U_{\xi}$ from $\mathbb{SP}^{\xi}_\mathbb{N}$ if $\mathbb{SP}^{\xi}_\mathbb{N}$ is a Spector class. Fix an order isomorphism $f\colon \alpha\to (D,R)$ and take $U=\{\langle n, m\rangle \mid m\in U_{f(n)} \}$. Then $(\mathbb{N};\mathbb{SP}^\alpha_\mathbb{N})$ and $(\mathbb{N};+,\times,\le,0,D,R,U)$ have the same definable sets.}%
, so we may think of elements of $\mathbb{SP}^\xi_\mathbb{N}$ as sets given by $<(1+\xi)$-fold iterated inductive definition.

Since members of $\mathbb{SP}(\mathfrak{M})$ are precisely $\Pi^1_1$ sets over $\mathfrak{M}$ (see Chapter 8 of \cite{Moschovakis1974ElementaryInd}),
we can derive the following characterization of members of $\mathbb{SP}^2_\mathbb{N}$:
\begin{lemma}
    The collection of all \emph{Arithmetical-in-a-$\Pi^1_1$-parameter formulas} is the least collection of formulas containing $\Pi^1_1$ formulas \emph{without set free variables} and closed under Boolean combinations and number quantification.
    
    A formula $\phi(x,Y)$ is \emph{$\Pi^1_1$-in-a-$\Pi^1_1$-parameter} if it has of the form
    $\forall X \phi(x,X,Y)$ for some arithmetical-in-a-$\Pi^1_1$-parameter formula $\phi(x,X,Y)$ with all free variables displayed. $\Pi^1_1(\Pi^1_1)$ is the collection of all $\Pi^1_1$-in-a-$\Pi^1_1$-parameter formulas.
    
    For $A\subseteq\mathbb{N}$, $A\in \mathbb{SP}^2_\mathbb{N}$ if and only if $A=\{n \mid \mathbb{N}\models \phi(n)\}$ for some $\Pi^1_1$-in-a-$\Pi^1_1$-parameter formula $\phi(x)$.
\end{lemma}

By the Barwise-Gandy-Moschovakis theorem (See Chapter 9 of \cite{Moschovakis1974ElementaryInd}), the members of $\mathbb{SP}^{\mu+1}_\mathbb{N}$ are exactly the $\Sigma_1^{\omega_{\mu+1}^\mathsf{CK}}$-definable subsets of $\mathbb{N}$.
Pohlers also defines the following:
\begin{equation*}
    \kappa^\mu_\mathbb{N}(T) := \sup\{\sigma_R(\vec{m})+1\mid R =\{\vec{m}\in \mathbb{N}^n \mid F(\vec{m})\}\in \mathbb{SP}^\mu_\mathbb{N},\ T\vdash F(\vec{m})\}.
\end{equation*}
Informally, $\kappa^\mu_\mathbb{N}(T)$ corresponds to the supremum of the closure ordinals of the members of $\mathbb{SP}^\mu_\mathbb{N}$ that are provably definable over $T$.

If $\omega^\mathsf{CK}_{\mu+1}$ admits a representation, which is always the case for recursive $\mu$, then we have $\kappa_\mathbb{N}^{\mu+1}(T) = |T|_{\Sigma_1^{\omega_{\mu+1}^\mathsf{CK}}}$. 
That is, our results concerning the latter also characterize the closure ordinals of iterated Spector classes.
Accordingly, this provides a proof-theoretic analysis of the members of iterated Spector classes, even though iterated Spector classes are \emph{semantically} defined.
For example, by applying the results in Section 6 to $\mu=2$, we can derive the following:
\begin{corollary}
    Let $\mathbf{\Gamma}=\Pi^1_1(\Pi^1_1)$ and $T$, $U$ be $\widehat{\mathbf{\Gamma}}$-definable $\mathbf{\Gamma}$-sound extensions of $\KP_1 + \Exists(\omega_2^\mathsf{CK})$, where $\widehat{\mathbf{\Gamma}}=\{\lnot\phi \mid \phi\in \mathbf{\Gamma}\}$. Then 
    \begin{equation*}
        T\subseteq^{\widehat{\mathbf{\Gamma}}}_{\mathbf{\Gamma}} U \iff |T|_\mathbf{\Gamma}\le |U|_\mathbf{\Gamma}.
    \end{equation*}
\end{corollary}

\subsection{Theories beyond $V=L$}

Theorem \ref{well-ordering-thm} demonstrates that a broad class of theories is well-ordered by a natural proof-theoretic comparison relation, namely, comparison of $\Sigma^\bfalpha_1$ theorems modulo a $\Pi^\bfalpha_1$ oracle. This is done by connecting this comparison relation to $\bfalpha$ ordinal analysis. By definition, $\bfalpha$ ordinal analysis measures how closely an axiom system coheres with G\"{o}del's constructible universe $L$. Thus, throughout this paper, we have assumed that $V=L$. We now turn to generalize Theorem \ref{well-ordering-thm} to axiomatic systems inconsistent with the assumption $V=L$. We are particularly interested in large cardinal axioms and other theories that prove the existence of $0^\sharp$.
Though we have assumed $V=L$ throughout our paper, the authors believe that very similar arguments should carry over to more complicated fine structural inner models allowing larger large cardinals.
The following proposition explains why this should be.
\begin{proposition}
    Let $\alpha$ be an admissible ordinal admitting a $\Sigma_1$ definition $\phi(x)$ and a $\Pi_1$ definition $\psi(x)$ such that $\mathsf{ZFC}^-$\footnote{$\ZFC^-$ is $\ZFC$ without power set but with Collection instead of Replacement.} proves $\exists x \phi(x)$ and $\phi(x)\leftrightarrow \psi(x)$.
    (In particular this applies in the case $\alpha=\omega_1^\mathsf{CK}$.) If $\vec{E}$ is an extender sequence, then $J^{\vec{E}}_\alpha = J_\alpha = L_\alpha$. 
\end{proposition}
\begin{proof}
    The inner model $J^{\vec{E}}_\alpha$ agrees with $J_\alpha$ below the first level $\lambda$ such that $J^{\vec{E}}_\lambda$ recognizes that the $\lambda$th extender $\vec{E}(\lambda)$ is not empty. Let $\kappa$ be its critical point.
    Then $J^{\vec{E}}_\lambda$ recognizes that $\kappa$ is a cardinal, so we get $J^{\vec{E}}_\kappa\models\ZFC^-$, whence $J^{\vec{E}}_\kappa\models \exists x \varphi(x)$. But $\varphi$ is $\Delta_1$ over $J^{\vec{E}}_\kappa=J_\kappa$.
    By upward absoluteness, $L[\vec{E}]\models \exists x \varphi(x)$, and since $\varphi$ is $\Delta_1$, $\varphi$ defines the same ordinal over $L$ as it does over $L[\vec{E}]$. Hence $\alpha<\kappa<\lambda$, and so $J_\alpha=J^{\vec{E}}_\alpha$. $J_\alpha=L_\alpha$ follows from the admissibility of $\alpha$.
\end{proof}

Accordingly, we can replace $L$ in our work from the previous sections with other fine-structural inner models, as long as we work with sufficiently small admissible ordinals. Thus, the well-foundedness and linearity phenomena generalize to theories that are compatible with other fine-structural inner models.

Here is a strong consequence of the above reasoning: All extensions of $\KP_1 + \Exists(\omega_1^{\mathsf{CK}})$ that are sound to some fine-structural inner model are linearly comparable with respect to $\Sigma_1^{\omega_1^{\mathsf{CK}}}$-reflection strength modulo $\Pi_1^{\omega_1^{\mathsf{CK}}}$-provability. Notice that this reasoning invokes the notion of $\Pi_1$ truth of formulas with respect to $L_{\omega_1^{\mathsf{CK}}}$, but not with respect to larger structures. Accordingly, one need only concede the determinacy of statements about $L_{\omega_1^{\mathsf{CK}}}$ to recognize the linearity phenomenon, even as it applies to large cardinals and axioms of higher set theory.

Our results may be adapted to prove a linear comparison result for higher classes of formulas. In an early draft of this paper, we conjectured that the linear comparison phenomenon generalizes to $\Pi^1_3$ for sufficiently strong theories. In private communication, Aguilera has informed us that this is indeed borne out. In particular, assuming $\mathsf{ZFC}$ + $\mathbf{\Delta}^1_2$-determinacy, the recursively axiomatized extensions of $\mathsf{ZFC}$ + $\mathbf{\Delta}^1_2$-determinacy are pre-linearly ordered by the relation $\subseteq^{\Sigma^1_3}_{\Pi^1_3}$. To indicate why we conjectured this fact in an early draft of our paper, let us recall the following well-known fact in inner model theory:
\begin{theorem}[Steel \cite{Steel1995ProjectivelyWOIM} Theorem 4.12]\label{steel} \pushQED{\qed}
    Let $M_n$ be the least inner model for $n$ Woodin cardinals, and let $\delta^1_n$ be the supremum of all $\Delta^1_n$-definable well-orders over a subset of $\omega$.
    Suppose that $M_n^\sharp$ exists, then
    \begin{enumerate}
        \item If $n$ is odd, then every $\Pi^1_{n+2}$-formula is equivalent to a $\Sigma_1^{M_n|\delta^1_{n+2}}$-formula.
        \item If $n$ is even, then every $\Sigma^1_{n+2}$-formula is equivalent to a $\Sigma_1^{M_n|\delta^1_{n+2}}$-formula. \qedhere 
    \end{enumerate}
\end{theorem}
Here $M_n|\alpha = (L_\alpha[\vec{E}],\vec{E}|\alpha)$ when $M_n=L[\vec{E}]$. For the case $n=1$, the presence of $M_1^\sharp$ implies that every $\Pi^1_3$-sentence is equivalent to a $\Sigma_1^{M_1|\delta^1_3}$-formula.
We conjecture that $\mathsf{ZFC}$ + $\mathbf{\Delta}^1_2$-determinacy proves that the following are equivalent for for ``sufficiently strong'' theories $T,U$ that prove the existence of $M_1^\sharp$:
\begin{enumerate}
    \item Either $T\subseteq_{\Pi^1_3}^{\Sigma^1_3}U$ or $U\subseteq_{\Pi^1_3}^{\Sigma^1_3}T$.
    \item Either $T\subseteq_{\Sigma_1^{M_1|\delta^1_3}}^{\Pi_1^{M_1|\delta^1_3}}U$ or $U\subseteq_{\Sigma_1^{M_1|\delta^1_3}}^{\Pi_1^{M_1|\delta^1_3}}T$.
\end{enumerate} 
This equivalence should follow from \autoref{steel}. Extensions of our methods should secure the second statement of the equivalence, thereby securing the first.

As a side note, the reader may wonder if the linear comparison phenomenon holds for $\Pi^1_{n+1}$-consequences or $\Sigma^1_n$-consequences for $n\ge 1$. The following result by Aguilera and Pakhomov provides a negative answer:
\begin{theorem}[\cite{AguileraPakhomov??Nonlinearity}] \pushQED{\qed}
    Working over $\ZFC$ with $\omega$ many Woodin cardinals, let $\Gamma$ be either one of $\Sigma^1_1$, $\Pi^1_2$, $\Sigma^1_3$, $\Pi^1_4$, $\cdots$, and let $\widehat{\Gamma}$ be the set of all negations of formulas in $\Gamma$.

    Suppose that $T$ is a sound recursively axiomatizable extension of $\mathsf{ACA}_0$. Then we can find true $\Gamma$-sentences $\phi_0$ and $\phi_1$ such that $T\nvdash^{\widehat{\Gamma}} \phi_0\to\phi_1$ and $T\nvdash^{\widehat{\Gamma}} \phi_1\to\phi_0$.\qedhere 
\end{theorem}

\bibliographystyle{plain}
\bibliography{bibliography}

\end{document}